\documentclass[11pt]{article}

\title{Topological and algebraic properties \\ of universal groups for right-angled buildings}
\author{%
	Jens Bossaert\thanks{Ghent University, \href{mailto:jens.bossaert@ugent.be}{\nolinkurl{jens.bossaert@ugent.be}}.}
	\and
	Tom De Medts\thanks{Ghent University, \href{mailto:tom.demedts@ugent.be}{\nolinkurl{tom.demedts@ugent.be}}.}
}
\date{January 27, 2021}

\usepackage{amsmath,amssymb,amsthm}
\usepackage[a4paper]{geometry}
\usepackage{tikz}
	\usetikzlibrary{shapes,calc,cd}
\usepackage{enumitem}
	\setenumerate[0]{label={\rm (\roman*)},leftmargin=*,itemsep=.2ex}
\usepackage[symbol]{footmisc}
\usepackage{stackengine}
\usepackage{hyperref}
	\hypersetup{bookmarksopen=true}
\usepackage[capitalize]{cleveref}

\renewcommand\footnotemark{}

\DeclareMathOperator{\Sym}{Sym}
\DeclareMathOperator{\Aut}{Aut}

\DeclareMathOperator{\U}{\mathcal U}
\newcommand{\F}{\boldsymbol F}
\renewcommand{\P}{\mathcal P}
\newcommand{\R}{\mathcal R}
\DeclareMathOperator{\Res}{Res}
\DeclareMathOperator{\proj}{proj}
\DeclareMathOperator{\dist}{dist}

\newcommand\llangle{\left\langle\!\left\langle}
\newcommand\rrangle{\right\rangle\!\right\rangle}

\newcommand\restrict[3][]{{#2\rvert_{#3}^{#1}}}

\let\norleq\trianglelefteq

\newcommand\acts{\mkern 1mu\relax.\mkern 1mu\relax}

\newtheorem{theorem}{Theorem}[section]
\newtheorem{lemma}[theorem]{Lemma}
\newtheorem{proposition}[theorem]{Proposition}
\newtheorem{corollary}[theorem]{Corollary}
\newtheorem{conjecture}[theorem]{Conjecture}

\newtheoremstyle{step}{\topsep}{\topsep}{\itshape}{0pt}{\itshape}{\@. }{0pt}{\thmname{#1}\thmnumber{ #2}\thmnote{ (#3)}}
\theoremstyle{step}
\newcounter{stepcounter}
\newtheorem{proofstep}[stepcounter]{Step}

\theoremstyle{definition}
\newtheorem{definition}[theorem]{Definition}
\newtheorem{remark}[theorem]{Remark}

\tikzstyle{diagramnode}=[circle,fill=black,minimum size=5pt,inner sep=0pt]
\tikzset{myvertex/.style={circle,minimum width=5pt,draw=none,fill=black,inner sep=0pt}}
\tikzset{every label/.style={rounded rectangle,font=\strut,inner sep=0pt}}

\newcommand{\DrawEllipse}[5][]{
	\draw[#1,rotate around={#5:#2}] #2 ellipse[x radius=#3,y radius=#4];
}
\newcommand{\DrawlabelledEllipse}[7][]{
	\draw[#1,rotate around={#5:#2}] #2 ellipse[x radius=#3,y radius=#4];
	\node[rounded rectangle,anchor={180+#7}] at ($#2+(#7:{#3*#4/veclen(#3*sin(#7-#5),#4*cos(#7-#5))})$) {#6};
}

\begin{document}

\maketitle

\begin{abstract}
	We study universal groups for right-angled buildings. Inspired by Simon Smith's work on universal groups for trees, we explicitly allow local groups that are not necessarily finite nor transitive. We discuss various topological and algebraic properties in this extended setting. In particular, we characterise when these groups are locally compact, when they are abstractly simple, when they act primitively on residues of the building, and we discuss some necessary and sufficient conditions for the groups to be compactly generated.
	
	We point out that there are unexpected aspects related to the geometry and the diagram of these buildings that influence the topological and algebraic properties of the corresponding universal groups.
\end{abstract}

\paragraph*{Keywords.} right-angled buildings, universal groups, locally compact groups, simple groups.
\paragraph*{MSC2020.} 51E24, 22F50, 22D05.


\section{Introduction}

In a foundational paper in 2000, Marc Burger and Shahar Mozes introduced the notion of a \emph{universal group} acting on a regular locally finite tree, with a local action prescribed by a fixed finite permutation group \cite{burgermozes2000}. When endowed with the permutation topology, the universal groups give interesting examples of totally disconnected, locally compact groups.

More precisely, let $X$ be a finite set, let $F\leq\Sym(X)$ be a permutation group and let $\mathcal T_d$ be a regular tree with valency $d=|X|$. The Burger--Mozes universal group $\U(F)$ is defined to be the group of all automorphisms of $\mathcal T_d$ having the property that for every vertex~$v$, the local action from the neighbourhood of $v$ to its image is given by a permutation in $F$ (in the sense of \cref{def:localaction}).

Since then, the definitions and results have been generalised in various directions. For example, it is rather straightforward to allow for two permutation groups $F_1$ and $F_2$ by following essentially the same construction using the biregular tree with the correct valencies. Also, one can relax the finiteness condition on the two groups and allow for permutation groups of infinite degree, although the topological properties of the resulting universal groups become more subtle in this setting. Recently, Simon Smith has used this approach to obtain the first construction of uncountably many pairwise non-isomorphic simple topological groups that are totally disconnected, locally compact, compactly generated, and non-discrete \cite{smith2017}.

Tom De Medts, Ana Silva and Koen Struyve pursued a different direction of generalisation and extended the notion of a universal group to locally finite \emph{right-angled buildings} \cite{demedts2018,demedts2019}. Right-angled buildings are geometric structures sharing some essential properties with trees (so that the notion of a universal group is still meaningful) while having a richer combinatorial structure. The resulting groups are, again, totally disconnected and locally compact.

The goal of this paper is to unify both generalisations. We define universal groups over right-angled buildings with local actions prescribed by arbitrary permutation groups (not necessarily transitive nor of finite degree). We collect some results on the algebraic and topological structure of these groups. In particular, we give precise characterisations for when they are locally compact, abstractly simple, or for when they act primitively on residues of the building. We also discuss necessary and sufficient conditions for when the universal groups are compactly generated. Interestingly, a couple of these properties not only depend on properties of the local groups, but also on the combinatorial structure of the diagram of the building.

Some results required new purely geometric arguments or constructions using right-angled buildings. In particular, we want to highlight \cref{lem:convexhullfinite}, in which we argue that any finite number of chambers in a right-angled building is contained in a finite convex subset, and \cref{th:implosion}, in which we introduce a controlled way of ``imploding'' right-angled buildings (in a very different way than a projection map does).

\paragraph*{Acknowledgment.}

We are extremely grateful to the anonymous referee, who did an impressive job reading our paper in detail, and made a variety of insightful observations, careful remarks and helpful suggestions. In particular, \cref{referee1}, \cref{referee2} and \cref{rem:fibres} would not have seen the light without the referee's feedback.

\paragraph*{Outline of the paper.}

In section~\ref{se:prelim}, we assemble the required preliminary definitions and results about topological groups, right-angled buildings and universal groups.
In section~\ref{se:props}, we collect some properties about the universal groups of right-angled buildings that do not require much more than extending earlier ideas to our new setting, and the results we obtain are natural and perhaps not very surprising. In particular, we characterise local compactness (\cref{prop:locallycompact}), compact generation of panel stabilisers (\cref{prop:stabcompgen}) and discreteness (\cref{prop:discrete}).

In section~\ref{se:compgen}, we deal with the more delicate problem of characterising compact generation. Our main result in this section is \cref{thm:compgen2}, providing a sufficient condition for compact generation. The converse is more delicate: the general case is left as a conjecture, but we discuss it thoroughly in the rest of this section.

Section~\ref{se:simp} deals with the simplicity of the universal groups. We obtain a complete characterisation in \cref{thm:simplicity}. The simplicity depends in an interesting way on the interplay between the diagram and the prescribed local actions. The proof relies on the new concept of implosion that we have mentioned above; see, in particular, \cref{th:implosion2}.

In section~\ref{se:prim}, finally, we investigate when the induced action on the set of residues of a given type is primitive. We provide a complete answer to this question in \cref{thm:prim}. Once again, this depends non-trivially on the diagram of the right-angled building.


\section{Preliminaries, notation, terminology}\label{se:prelim}

\subsection{Topological groups}

\begin{definition}[permutation topology]
	Consider a permutation group $G\leq\Sym(X)$. We give $G$ the structure of a topological group by taking as a neighbourhood basis of the identity the set of pointwise stabilisers of finite subsets of $X$.
\end{definition}

If we interpret $X$ as a discrete topological space and elements of $G$ as maps $X\to X$, then the permutation topology on $G\leq\Sym(X)$ agrees with the topology of pointwise convergence.

Since we will frequently deal with pointwise and setwise stabilisers, it is desirable to fix a clear notation. For a group $G$ acting on a set $X$, we will denote the pointwise stabiliser of some subset $Y\subseteq X$ as 
\[G_{(Y)} = \{g\in G\mid \text{$g\acts y = y$ for all $y\in Y$}\},\]
and the setwise stabiliser as
\[G_{\{Y\}} = \{g\in G\mid \text{$g\acts y \in Y$ for all $y\in Y$}\}.\]
Whenever $Y=\{y\}$, we will simply write the stabiliser as $G_y$ since both notions obviously agree. However, in order to avoid potential confusion, we will never use an abbreviated notation like $G_Y$ when $Y$ is not a singleton.

For a good overview of the interplay between permutational and topological properties of permutation groups, we refer the reader to \cite{moller10} or \cite{woess91}. In particular, we will make frequent use of the following facts.

\begin{proposition}\label{prop:permtop}
	Let $G\leq\Sym(X)$ be endowed with the permutation topology.
	\begin{enumerate}
		\let\emph\relax
		\item $G$ is always \emph{Hausdorff} and \emph{totally disconnected}.
		\item $G$ is \emph{open} if and only if $G$ contains the pointwise stabiliser of a finite subset of $X$.
		\item $G$ is \emph{closed} if and only if some point stabiliser is closed, which holds if and only if all point stabilisers are closed.
		\item $G$ is \emph{discrete} if and only if $X$ contains a finite subset with trivial pointwise stabiliser.
		\item $G$ is \emph{compact} if and only if $G$ is closed and all orbits of $G$ are finite.
		\item If $G$ is closed, then $G$ is \emph{locally compact} if and only if $X$ contains a finite subset with compact pointwise stabiliser.
	\end{enumerate}
\end{proposition}

\subsection{Buildings}

We will follow the combinatorial approach to buildings and assume familiarity with most of the basic notions of building theory. Unless explicitly mentioned otherwise, we adopt the definitions and notations from the books \cite{ronan} by Mark Ronan and \cite{weiss2003} by Richard Weiss.

From now on, $\Delta$~will always denote a building of type $M=(m_{ij})$ over a finite index set $I$. Recall that the \emph{rank} of $\Delta$ is by definition the cardinality of $I$. We view $\Delta$ as a chamber system, so we will draw chambers as vertices and connect $i$-adjacent chambers by edges labelled $i$. We write $i$-adjacency as $\sim_i$ (or simply as $\sim$ when the type $i$ is irrelevant). A \emph{gallery} is a path of consecutive chambers
	\[c_0 \sim_{i_1} c_1 \sim_{i_2} \dots \sim_{i_k} c_k\]
	and its \emph{type} is the sequence $(i_1, i_2,\dots, i_k)$. For a subset $J\subseteq I$, a \emph{$J$-residue} is a maximal subset of chambers of $\Delta$ that can be connected by some gallery whose type contains only elements in $J$. An $\{i\}$-residue is simply called an \emph{$i$-panel} and is then a maximal subset of pairwise $i$-connected chambers. We will usually denote a panel or a residue by $\mathcal P$ or $\mathcal R$, respectively. The set of all $J$-residues of $\Delta$ will be denoted by $\Res_J(\Delta)$. We usually identify the building and its residues with the sets of their chambers.
	
	A (type-preserving) automorphism of $\Delta$ is a bijective map on the chamber set that preserves $i$-adjacency for all $i\in I$. We denote the group of all automorphisms by $\Aut(\Delta)$.

	For basic results about residues and for related notions (such as minimal galleries, elementary homotopies, projections, and convexity), we refer to \cite{ronan} or \cite{weiss2003}.

\begin{definition}[right-angled]
	The building $\Delta$ is called \emph{right-angled} if either $m_{ij}=2$ or $m_{ij}=\infty$ for all distinct $i,j\in I$, or in other words, if the Coxeter diagram of $\Delta$ only has edges labelled $\infty$.
\end{definition}

In particular, a tree $\mathcal T$ can be seen as a right-angled building of rank $2$: the chambers are the edges of $\mathcal T$, and the two types of adjacency are defined by sharing a vertex in one of the two bipartition classes. A vertex of $\mathcal T$ then corresponds to a panel of the building. Note that this structure, viewed as a chamber system, would actually be drawn as the line graph of $\mathcal T$.

\begin{definition}[semiregular]
	The building $\Delta$ is \emph{semiregular with parameters $(q_i)_{i\in I}$} if all $i$-panels have the same (possibly infinite) cardinality $q_i\geq 2$.
\end{definition}

In sharp contrast with the case of spherical buildings, semiregular right-angled buildings have essentially no restrictions on their parameters. The following theorem is due to Fr\'ed\'eric Haglund and Fr\'ed\'eric Paulin.

\begin{theorem}[Haglund--Paulin]
	For any choice of (possibly infinite) cardinal numbers $(q_i)_{i\in I}$ with each $q_i\geq 2$, there exists a semiregular right-angled building with these parameters. Moreover, this building is unique up to isomorphism.
\end{theorem}
\begin{proof}
	See \cite[Proposition 1.2]{haglundpaulin03}.
\end{proof}

We will mostly focus on irreducible buildings.
\begin{definition}[irreducible]
	The building $\Delta$ is called \emph{reducible} if the index set~$I$ is a disjoint union of two non-empty subsets $I = I_1\sqcup I_2$ such that $m_{i_1i_2}=2$ for every $i_1\in I_1$ and every $i_2\in I_2$. It is called \emph{irreducible} otherwise. In other words, a building is irreducible if its Coxeter diagram is connected (as a graph).
\end{definition}
Notice that a reducible building can be decomposed as a direct product $\Delta\cong\Delta_1\times\Delta_2$ where $\Delta_1$ and $\Delta_2$ are right-angled buildings of type $I_1$ and $I_2$, respectively.

An important combinatorial property of right-angled buildings is the following lemma, that in certain circumstances allows us to ``complete squares'' in the building.
\begin{lemma}[closing squares]
	\label{lem:closingsquares}
	Let $c_0$ be a fixed chamber in a right-angled building $\Delta$.
	\begin{enumerate}
		\item Let $c,d_1,d_2\in\Delta$ be such that $\dist(c_0,c)=n+1$, $\dist(c_0,d_1)=\dist(c_0,d_2)=n$, and $d_1\sim_i c\sim_j d_2$ for some $i\neq j\in I$. Then $m_{ij}=2$ and there exists a chamber~$e$ such that $\dist(c_0,e)=n-1$ and $d_1\sim_j e\sim_i d_2$.
		\item Let $c_1,c_2,d_1\in\Delta$ be such that $\dist(c_0,c_1)=\dist(c_0,c_2)=n+1$, $\dist(c_0,d_1)=n$ and $d_1\sim_i c_1\sim_j c_2$ for some $i\neq j\in I$. Then $m_{ij}=2$ and there exists a chamber~$d_2$ such that $\dist(c_0,d_2)=n$ and $d_1\sim_j d_2\sim_i c_2$.
	\end{enumerate}
	
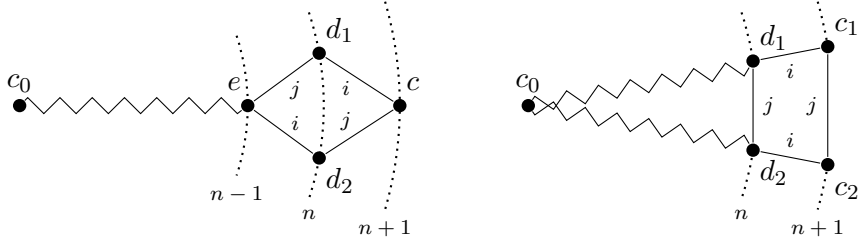
\[\begin{tikzpicture}
		\node[myvertex,label=above:$c_0$] (C0) {};
		\draw[dotted,thick] (18:3) arc (18:-18:3) node[font=\scriptsize,below] {$n-1$}
				(0:3) node[myvertex,label=150:$e$] (E) {}
			(18:4) arc (18:-18:4) node[font=\scriptsize,below] {$n$}
				(10:4) node[myvertex,label=30:$d_1$] (D1) {}
				(-10:4) node[myvertex,label=330:$d_2$] (D2) {}
			(16:5) arc (16:-16:5) node[font=\scriptsize,below] {$n+1$}
				(0:5) node[myvertex,label=30:$c$] (C) {};
		\draw (C0) --++ (45:.15) \foreach\i in {1,...,6} {--++ (-45:.3) --++ (45:.3)} --++ (-45:.3) -- (E);
		\draw[every node/.style={font=\scriptsize}]
			(E) to node[below,pos=.7] {$j$} (D1)
				to node[below,pos=.3] {$i$} (C)
				to node[above,pos=.7] {$j$} (D2)
				to node[above,pos=.3] {$i$} (E);
	\end{tikzpicture}
	\qquad
	\begin{tikzpicture}
		\node[myvertex,label=above:$c_0$] (C0) at (1,0) {};
		\draw[dotted,thick] (18:4) arc (18:10:4)
			(-10:4) arc (-10:-18:4) node[font=\scriptsize,below] {$n$}
				(8.5:4) node[myvertex,label=30:$d_1$] (D1) {}
				(-8.5:4) node[myvertex,label=330:$d_2$] (D2) {}
			(16:5) arc (16:9:5) (-9:5) arc (-9:-16:5) node[font=\scriptsize,below] {$n+1$}
				(9:5) node[myvertex,label=30:$c_1$] (C1) {}
				(-9:5) node[myvertex,label=330:$c_2$] (C2) {};
		\draw (C0) --++ (55:.15) \foreach\i in {1,...,6} {--++ (-35:.3) --++ (55:.3)} --++ (-35:.3) -- (D1);
		\draw (C0) --++ (-55:.15) \foreach\i in {1,...,6} {--++ (35:.3) --++ (-55:.3)} --++ (35:.3) -- (D2);
		\draw[every node/.style={font=\scriptsize}]
			(D1) to node[below] {$i$} (C1)
				to node[left] {$j$} (C2)
				to node[above] {$i$} (D2)
				to node[right] {$j$} (D1);
	\end{tikzpicture}\]
\end{lemma}
\begin{proof}
	We refer to Lemma 2.9 and Lemma 2.10 in \cite{demedts2018}.
\end{proof}

Finally, for some more technical lemmas (\cref{lem:parallellocals}, \cref{lem:extendlocaluniversal}, \cref{lem:welldefinedicolours}) we will need the notion of parallelism of panels from \cite{caprace14}. We recall some definitions and crucial results. Note that \cite{caprace14} considers parallelism in a more general sense, for arbitrary residues; for our purposes, parallelism of panels will suffice.
\begin{definition}[parallelism]
	Panels $\P_1$ and $\P_2$ are called \emph{parallel} if $\proj_{\P_1}(\P_2) = \P_1$ and $\proj_{\P_2}(\P_1) = \P_2$.
\end{definition}
\begin{definition}
	For $k\in I$, we define $k^\perp = \{i\in I\setminus\{k\} \mid ik=ki\} = \{i\in I \mid m_{ik}=2\}$.
\end{definition}
\begin{proposition}
\label{prop:parallelism}
	Let $\P$ and $\P'$ be two panels.
	\begin{enumerate}
		\item If two chambers of $\P'$ have distinct projections on $\P$, then $\P$ and $\P'$ are parallel.
		\item Let $k$ be the type of $\P$. Then $\P'$ is parallel to $\P$ if and only if $\P'$ is of type $k$ and $\P$ and $\P'$ are contained in a common residue of type $k\cup k^\perp$.
		
		In particular, parallel panels have the same type.
	\end{enumerate}
\end{proposition}
\begin{proof}
	Property (i) is \cite[Proposition 2.5]{caprace14}, (ii) is \cite[Proposition 2.8]{caprace14}.
\end{proof}

\subsection{Universal groups}

In order to keep track of the local actions, we introduce the notion of a \emph{coloured building}. In what follows, let $\Delta$ be a semiregular right-angled building with parameters $(q_i)_{i\in I}$.

\begin{definition}[colourings]\label{def:colourings}
	For every $i\in I$, let $X_i$ be a set of \emph{$i$-colours} (or \emph{$i$-labels}) such that $|X_i|=q_i$. A \emph{legal colouring} of $\Delta$ is a map
\[\lambda \colon \Delta \to \prod_{i\in I} X_i \colon c \mapsto \big(\lambda_1(c),\dots,\lambda_n(c)\big)\]
satisfying the following properties for every $i\in I$ and for every $i$-panel $\P$:
\begin{enumerate}
	\item the restriction $\restrict{\lambda_i}{\P} \colon \P\to X_i$ is a bijection;
	\item for every $j\neq i$, the restriction $\restrict{\lambda_j}{\P} \colon \P\to X_j$ is constant.
\end{enumerate}
\end{definition}

The following result is basic but important.
\begin{proposition}
	\label{prop:recolouring}
	Let $\lambda$ and $\lambda'$ be two different legal colourings of $\Delta$. Let $c$ and $c'$ be two chambers such that $\lambda(c)=\lambda'(c')$. Then there exists an automorphism $g\in\Aut(\Delta)$ such that $g\acts c=c'$ and $\lambda'=\lambda\circ g$.
\end{proposition}
\begin{proof}
	We refer to \cite[Proposition 2.44]{demedts2018}.
\end{proof}

Using colourings we now come to the definitions of local actions and universal groups.
\begin{definition}[local action]
	\label{def:localaction}
	Let $\lambda$ be a legal colouring of $\Delta$. Consider an $i$-panel~$\P$ and a building automorphism $g\in\Aut(\Delta)$. The \emph{local action} of $g$ at $\P$ is the map
	\[\sigma_\lambda(g,\P) = \restrict{\lambda_i}{g\acts\P} \circ \restrict{g}{\P} \circ \restrict[-1]{\lambda_i}{\P}\]
	which is an element of $\Sym(X_i)$ by definition of $\lambda$.
\end{definition}

\begin{definition}[universal group]\label{def:universal}
	Let $\F$ be a map that associates to every $i\in I$ some permutation group $F_i\leq\Sym(X_i)$. Let $\Delta$ be a semiregular right-angled building with parameters $q_i=|X_i|$. Fix a legal colouring $\lambda$ of $\Delta$. We define the \emph{universal group} of $\Delta$ 
	w.r.t.~the data $\F$ as
	\[\U_\Delta^\lambda(\F) = \left\{g\in\Aut(\Delta) \mid \text{$\sigma_\lambda(g,\P) \in F_i$ for every $i\in I$ and $\P\in\Res_i(\Delta)$}\right\},\]
	i.e., as the subgroup of all automorphisms that locally act like permutations in~$F_i$. We call the $F_i$ the \emph{local groups}. Note that we do not require the local groups to be transitive, nor to have finite degree.
\end{definition}

It is not hard to see that the local actions satisfy
\[\sigma_\lambda(gh,\P) = \sigma_\lambda(g,h\acts\P)\cdot\sigma_\lambda(h,\P)
	\qquad\text{and}\qquad
	\sigma_\lambda(g^{-1},\P) = \sigma_\lambda(g,g^{-1}\acts\P)^{-1}\]
so it readily follows that every universal group is indeed a subgroup of $\Aut(\Delta)$. Moreover, the choice of the legal colouring is not essential for the structure of the universal group, in the following sense.
\begin{proposition}
	For different choices of legal colourings $\lambda$ and $\lambda'$, the corresponding universal groups $\U_\Delta^\lambda(\F)$ and $\U_\Delta^{\lambda'}(\F)$ are conjugate in $\Aut(\Delta)$.
\end{proposition}
\begin{proof}
	We refer to \cite[Proposition 3.7 (i)]{demedts2018}.
\end{proof}
When irrelevant or clear from the context we usually suppress the dependency on $\Delta$ and the colouring $\lambda$ in the notation, and e.g.~abbreviate $\U_\Delta^\lambda(\F)$ as $\U(\F)$ and $\sigma_\lambda$ as~$\sigma$. Moreover, we will frequently omit the local data $\F$ in our statements and further abbreviate $\U(\F)$ as $\U$ when there is no ambiguity possible.

\begin{definition}
	For each panel $\mathcal P$, the \emph{panel group} $\restrict{\U}{\P}$ is the subgroup of $\Sym(\P)$ induced by the action of the panel stabiliser $\U_{\{\P\}}$ on the chambers of $\P$.
\end{definition}

We conclude with a definition that will turn out to be convenient throughout this article; see, in particular, \cref{lem:orbits} below.
\begin{definition}[harmony]
	Let $J$ and $K$ be two disjoint subsets of $I$. We will call $K$-residues $\R$ and $\R'$ \emph{$J$-harmonious} if for each $j\in J$, the two (well-defined) colours $\lambda_j(\R)$ and $\lambda_j(\R')$ lie in the same orbit of $F_j$.
	
	By slight abuse of notation, for $J=\{j\}$, we abbreviate $\{j\}$-harmony as $j$-harmony. Moreover, for $J=I\setminus K$, we abbreviate $(I\setminus K)$-harmony simply as harmony.

	In particular, for $K=\emptyset$, two chambers $c$ and $c'$ are harmonious if for each $i\in I$, the colours $\lambda_i(c)$ and $\lambda_i(c')$ lie in the same orbit of $F_i$.
\end{definition}


\section{Properties of the universal groups}\label{se:props}

As a first result, we can quite easily calculate the orbits of $\U$ on chambers and residues.
\begin{proposition}
\label{lem:orbits}
	Two residues lie in the same orbit of $\U$ if and only if they are harmonious.
	In particular, two chambers $c$ and $c'$ lie in the same orbit of $\U$ if and only if their colours $\lambda_i(c)$ and $\lambda_i(c')$ lie in the same orbit of $F_i$ for each $i\in I$.
\end{proposition}
\begin{proof}
	First consider two harmonious $J$-residues $\R$ and $\R'$. For each $i\notin J$, let $f_i\in F_i$ be a permutation such that $f_i\acts\lambda_i(\R)=\lambda_i(\R')$. For each $i\in J$, let $f_i\in F_i$ be the identity on $X_i$. Define a ``recolouring map''
	\[\phi\colon\prod_{i\in I} X_i\to\prod_{i\in I} X_i\colon \big(x_1,\dots,x_n\big)\mapsto\big(f_1\acts x_1,\dots,f_n\acts x_n\big).\]
	It is clear from \cref{def:colourings} that $\phi\circ\lambda$ is again a legal colouring. Take any chamber $c\in\R$ and let $c'\in\R'$ be a chamber with the same $J$-colours. Then $\lambda(c)=(\phi\circ\lambda)(c')$ and \cref{prop:recolouring} provides a building automorphism $g\in\Aut(\Delta)$ such that $g\acts c=c'$ and $\phi\circ\lambda=\lambda\circ g$. In particular, $g(\R)=\R'$.
	
	We claim that $g\in \U = \U^\lambda(\F)$, so that $\R$ and $\R'$ will lie in the same orbit of $\U$. In order to show this, let $i\in I$ and $d\in\Delta$, and let $\P$ be the $i$-panel containing $d$. Then
	\[\sigma_\lambda(g,\P)
		=\restrict{\lambda_i}{g\acts\P} \circ \restrict{g}{\P} \circ \restrict[-1]{\lambda_i}{\P}
		=f_i \circ \restrict{\lambda_i}{\P} \circ \restrict[-1]{\lambda_i}{\P}
		=f_i\in F_i\]
	using the fact that $\phi\circ\lambda=\lambda\circ g$. Our claim is proved.
	
	\smallskip
	Conversely, suppose that $g\acts\R=\R'$ for some $g\in\U$ and let $i\notin J$. Let $c$ be an arbitrary chamber in $\R$ and let $\P$ be the $i$-panel containing $c$. Note that $\lambda_i(\R)=\lambda_i(c)$ by definition of legal colourings. If we now define
	\[f_i=\sigma_\lambda(g,\P)
		\in F_i ,\]
	then clearly
	\[f_i\acts\lambda_i(\R)
		=f_i\acts\lambda_i(c)
		=\left(\restrict{\lambda_i}{g\acts\P} \circ \restrict{g}{\P} \circ \restrict[-1]{\lambda_i}{\P} \circ \lambda_i\right)(c)
		=\lambda_i(g\acts c) 
		=\lambda_i(\R').\]
	Hence for each $i\notin J$, the colours $\lambda_i(\R)$ and $\lambda_i(\R')$ lie in the same orbit of $F_i$.
\end{proof}

An immediate corollary is the following.
\begin{corollary}
	The action of $\U(\F)$ on $\Delta$ has finite orbits if and only if all local actions have finite orbits. More precisely,
	\[\left\lvert \Delta/\U\right\rvert = \prod_{i\in I} \left\lvert X_i/F_i\right\rvert.\]
	In particular the group $\U(\F)$ is transitive if and only if all local groups $F_i$ are transitive.
\end{corollary}

It is worth explicitly mentioning the following observation.
\begin{lemma}
	For every $i\in I$, let $F_i\leq F_i'\leq\Sym(X_i)$. Let $\F$ and $\F'$ be the corresponding local data. Then there is a natural inclusion of $\U_\Delta^\lambda(\F)$ into $\U_\Delta^\lambda(\F')$.
\end{lemma}
\begin{proof}
	This follows immediately from \cref{def:universal}.
\end{proof}
\begin{corollary}
	\label{cor:cobounded}
	The action of $\U(\F)$ on $\Delta$ is cobounded. More precisely, every ball in $\Delta$ of radius~$I$ contains a representative chamber of every orbit of $\U(\F)$.
\end{corollary}
\begin{proof}
	Let $\mathbf{Id}$ map every $i\in I$ to the trivial subgroup of $\Sym(X_i)$. By \cref{lem:orbits}, two chambers $c, c'$ lie in the same orbit of $\U(\mathbf{Id})$ if and only if $\lambda_i(c) = \lambda_i(c')$ for all $i \in I$.
	We claim that the action of $\U(\mathbf{Id})$ is cobounded. Indeed, let $c$ be any chamber, let $x_i\in X_i$ be arbitrary colours, and let $J=\{i\in I \mid \lambda_i(c_0)\neq x_i\}$. Then there exists a gallery of length $|J|$ that joins $c$ to a chamber $d$ with $\lambda_i(d)=x_i$ for all $i\in I$ (in fact, there exists a gallery of every type obtained by multiplying the elements of $J$ in some order). Since $d$ lies in the ball of radius $I$ around $c$, this proves our claim.
	
	Since $\U(\mathbf{Id}) \leq \U(\F)$, this implies that the action of $\U(\F)$ is cobounded as well.
\end{proof}

We also have precise information about the panel groups.
\begin{lemma}
	\label{localpermiso}
	Let $i\in I$ and let $\P$ be an $i$-panel. Then the panel group $\restrict{\U}{\P} \leq \Sym(\P)$ is permutationally isomorphic to $F_i$.
\end{lemma}
\begin{proof}
	We refer to \cite[Lemma 3.5]{demedts2018}.
\end{proof}

In the terminology of \cite{caprace14}, the following lemma states that finite sets of chambers are contained in wings, and moreover, we have control over the colour of the basepoint.
\begin{lemma}
\label{lem:finiteinwing}
	Let $\Phi$ be a finite set of chambers, let $i \in I$ and let $x\in X_i$ be an $i$-colour. Assume that the diagram of $\Delta$ has no isolated nodes. Then there exists an $i$-panel $\P$ such that $\proj_\P(\Phi)$ is a single chamber $c$ in $\P$ with \mbox{$\lambda_i(c)=x$}.
\end{lemma}
\begin{proof}
	Let $j\in I$ be such that $m_{ij}=\infty$ and let $\R$ be an arbitrary residue of type $\{i,j\}$ (which is a tree). The projection of $\Phi$ onto $\R$ defines a finite set of chambers $\Phi'$ in~$\R$, which can be enclosed by a ball $B$ of finite diameter. Let $\P_0$ be any $i$-panel in $\R\setminus B$.
	
	Since $\R$ is a tree, the projection $\proj_{\P_0}(\Phi')$ is just a single chamber $c_0$. If $\lambda_i(c_0)=x$, then we let $c=c_0$. If $\lambda_i(c_0)\neq x$, then let $c$ be any chamber $j$-adjacent to the (unique) chamber in $\P_0$ with $i$-colour $x$. In both cases, $\lambda_i(c)=x$.
	
	Let $\P$ be the $i$-panel containing~$c$; we claim that $\proj_\P(\Phi) = \{c\}$. Let $d\in\Phi$ and let $\gamma_d$ be the type of a minimal gallery from $d$ to its projection onto $\R$. By definition of projections $\gamma_d$ cannot be homotopic to a word ending in $i$~or~$j$. Therefore the gallery of type $\gamma_d\:(ij)^k$ or $\gamma_d\: j(ij)^k$ from $d$ along $\proj_{\R}(d)$ to $c$ is minimal, and we conclude that $\proj_\P(d)=c$.
\end{proof}

Another key ingredient is \cref{lem:extendlocaluniversal} below, building further upon results by Pierre-Emmanuel Caprace \cite[Proposition 4.2]{caprace14} and Tom De Medts, Ana Silva and Koen Struyve \cite[Proposition 3.15, Theorem 4.7]{demedts2018}. The proposition allows to extend local permutations to universal group elements in a controlled way.

\begin{lemma}
\label{lem:parallellocals}
	Let $g$ be an automorphism of $\Delta$ and let $\P$ and $\P'$ be two parallel $k$-panels. Then the local actions $\sigma_\lambda(g,\P)$ and $\sigma_\lambda(g,\P')$ are identical permutations in $F_k$.
\end{lemma}
\begin{proof}
	By \cref{prop:parallelism} (ii), $\P$ and $\P'$ are contained in a residue of type $k\cup k^\perp$. This residue has a direct product structure $\P_0\times\R_0$, where $\P_0$ is a $k$-panel and $\R_0$ a residue of type $k^\perp$. Consider any chamber $c\in\P$. By this direct product structure, the unique minimal gallery from $c$ to $\proj_{\P'}(c)$ is contained in a residue of type $k^\perp$. In particular, it follows that $\lambda_k(c) = \lambda_k(\proj_{\P'}(c))$ for all $c\in\P$. Moreover, as $g$ is an automorphism, $g\acts\proj_{\P'}(c) = \proj_{g\acts\P'}(g\acts c)$ for all $c\in\P$.
	
	Now consider the $k$-colour $\lambda_k(g\acts c)$. On the one hand, by definition of local actions,
	\[\lambda_k(g\acts c)
		= \sigma_\lambda(g,\P)\acts\lambda_k(c).\]
	On the other hand, using the projection onto $\P'$, we find that
	\begin{align*}
		\lambda_k(g\acts c)
		& = \lambda_k(\proj_{g\acts\P'}(g\acts c))\\
		& = \lambda_k(g\acts \proj_{\P'}(c))\\
		& = \sigma_\lambda(g,\P')\acts\lambda_k(\proj_{\P'}(c))\\
		& = \sigma_\lambda(g,\P')\acts\lambda_k(c).
	\end{align*}
	Since this holds for all $c\in\P$, we conclude that $\sigma_\lambda(g,\P) = \sigma_\lambda(g,\P')$.
\end{proof}
Note that we can bundle the information of \cref{lem:parallellocals} in a commutative diagram:
\[\begin{tikzcd}
	& g\acts c \ar[dd,"\lambda_k",pos=.75] \ar[rr,"\proj_{g\acts\P'}"] &&
		g\acts c' \ar[dd,"\lambda_k",pos=.75]\\
	c \ar[rr,"\proj_{\P'}",pos=.75,crossing over] \ar[ur,"g"] \ar[dd,"\lambda_k",pos=.75] &&
		c' \ar[ur,"g"] &\\
	& \lambda_k(g\acts c) \ar[rr,equal] &&
		\lambda_k(g\acts c') \\
	\lambda_k(c) \ar[ur,right,"{\sigma_\lambda(g,\P)}",swap] \ar[rr,equal] &&
		\lambda_k(c') \ar[ur,right,"{\sigma_\lambda(g,\P')}",swap] \ar[uu,leftarrow,"\lambda_k",pos=.25,swap,crossing over] &
\end{tikzcd}\]

\begin{proposition}
\label{lem:extendlocaluniversal}
	Let $\P_0$ be a $k$-panel, let $f_0\in F_k$, and let $\lambda$ be a legal colouring of~$\Delta$. Then there exists an automorphism $g\in\U_\Delta^\lambda(\F)$ such that
	\begin{enumerate}
		\item $g$ stabilises $\P_0$,
		\item the local action $\sigma_\lambda(g,\P_0)$ is equal to $f_0$,
		\item $g$ fixes all chambers $c$ with the property that $\lambda_k(\proj_{\P_0}(c))$ is fixed by $f_0$.
	\end{enumerate}
\end{proposition}
\begin{proof}
	First, for each $i\in I$ and each pair $(x,y)$ of $i$-colours in the same orbit of $F_i$, we let $\varphi_i(x,y)$ be a fixed element of $F_i$ such that $\varphi_i(x,y)\acts x = y$. When $x=y$, we explicitly choose $\varphi_i(x,y)$ equal to the identity permutation. We will construct the automorphism $g\in\U(\F)$ in such a way that all the local actions equal either $f_0$ or some $\varphi_i(x,y)$, and we do this by induction on the distance to the $k$-panel $\P_0$.
	
	For a chamber $c$ such that $\dist(c,\P_0) = 0$, i.e.~$c\in\P_0$, let $g\acts c$ be the unique chamber in $\P_0$ such that $\lambda_i(g\acts c) = f_0\acts \lambda_i(c)$. This guarantees that $\sigma_\lambda(g,\P_0) = f_0$.
	
	Next, let
	\[B_n = \{d\in\Delta \mid \dist(d,\P_0) \leq n\}\]
	be the ``ball'' centered in the panel $\P_0$ of radius $n$. Assume by induction that $g$ is defined on $B_n$, the local actions of $g$ in $B_n$ are either $f_0$ or some $\varphi_i(x,y)$, and for each $d\in B_n$, the chamber $g\acts d$ is harmonious to $d$. In order to extend $g$ to $B_{n+1}$, let $c$ be a chamber such that $\dist(c,\P_0) = n+1$. Define the set
	\[D(c) = \{d\in\Delta \mid \text{$\dist(d,\P_0)=n$ and $d\sim c$}\} \subset B_n.\]
	Let $m = |D(c)|$. We consider two cases.
	\begin{enumerate}
		\item If $m\geq 2$, then we can invoke \cref{lem:closingsquares} to conclude that $D(c) \cup \{c\}$ is contained in an apartment $\mathcal A$ of a spherical residue of rank $m$, where $\mathcal A\setminus\{c\} \subset B_n$. Hence, the image $g\acts c$ is determined by the image of $\mathcal A\setminus\{c\}$.
		\item If $m=1$, then we have some freedom in defining $g\acts c$. Let $d\sim_i c$ be the unique chamber in $D(c)$ and consider the colours $x=\lambda_i(d)$ and $y=\lambda_i(g\acts d)$. By induction hypothesis, $x$ and $y$ are in the same orbit of $F_i$. We now define $g\acts c$ as the unique chamber $i$-adjacent to $g\acts d$ such that $\lambda_i(g\acts c) = \varphi_i(x,y)\acts\lambda_i(c)$.
	\end{enumerate}
	This extends $g$ to $B_{n+1}$. In both cases it should be clear that $c$ and $g\acts c$ are harmonious. We now show in more detail that all local actions in $B_{n+1}$ are as we wanted, i.e.~either $f_0$ or $\varphi_i(x,y)$. Consider an arbitrary panel $\P$ of type $i$, completely contained in $B_{n+1}$. We consider three cases.
	\begin{enumerate}
		\item If $\P$ is contained in $B_n$, then $\sigma_\lambda(g,\P)$ is $f_0$ or $\varphi_i(x,y)$ by the induction hypothesis.
		\item If $\P$ is not contained in $B_n$ and there is a unique chamber $d$ in $\P$ closest to $\P_0$, then $\dist(d,\P_0) = n$. Note that the number $|D(c)|$ is constant for chambers $c\in\P$ different from $d$. If $|D(c)|=1$, then the local action $\sigma_\lambda(g,\P)$ is equal to $\varphi_i(x,y)$ with $x = \lambda_i(d)$ and $y = \lambda_i(g\acts d)$. If $|D(c)|\geq 2$, then by \cref{lem:closingsquares}, $\P$ is parallel to a panel in $B_n$. The conclusion then follows from \cref{lem:parallellocals} and the induction hypothesis.
		\item Finally if $\P$ is not contained in $B_n$ and there is no unique chamber in $\P$ closest to~$\P_0$, then $\P_0$ and $\P$ are parallel by \cref{prop:parallelism} (i). It follows immediately from \cref{lem:parallellocals} that $\sigma_\lambda(g,\P) = f_0$.
	\end{enumerate}
	In conclusion, we can extend $g$ to a automorphism of the whole building, such that all local actions are either $f_0$ or some $\varphi_i(x,y)$, and in particular $g\in\U_\Delta^\lambda(\F)$.
	
	It remains to show the desired property (iii) of this automorphism. Let $c_0\in\P_0$ be a chamber that is fixed by $g$. To show that $g$ fixes all chambers $c$ such that $c_0=\proj_{\P_0}(c)$, we will use induction on the distance $\dist(c,c_0)$. For $c=c_0$ there is nothing to show. Now let $\dist(c,c_0)=n+1$ and assume that $g$ fixes all chambers in $D(c)$. If $|D(c)|\geq 2$, then also $c$ is fixed by construction. If $|D(c)|=1$, say $D(c)=\{d\}$ with $d\sim_i c$, then the local action of $g$ on the $i$-panel containing $c$ is by construction equal to $\varphi_i(x,y)$, with $x=\lambda_i(d)$ and $y=\lambda_i(g\acts d)=\lambda_i(d)=x$. This local action is the identity, so $g$ fixes not only $d$ but the whole $i$-panel. In any case we find that $g$ fixes $c$, proving our claim.
\end{proof}

The proof of the following lemma is completely similar to the proof of~\cite[Lemma~7]{smith2017} but it is worth repeating here in the setting of right-angled buildings.
\begin{lemma}
	\label{lem:local-cont}
	Let $i\in I$ and $\P\in\Res_i(\Delta)$, and consider the map
	\[\sigma_{\lambda,\P} \colon \Aut(\Delta)\to\Sym(X_i) \colon g\mapsto\sigma_\lambda(g,\P).\]
	Then $\sigma_{\lambda,\P}$ is continuous w.r.t.~the permutation topologies on $\Aut(\Delta)$ and $\Sym(X_i)$.
\end{lemma}
\begin{proof}
	We already know that $\sigma_{\lambda,\P}$ is surjective. Let $W \subseteq \Sym{X_i}$ be an open subset and consider any $g \in \sigma_{\lambda,\P}^{-1}(W) \subseteq \Aut{\Delta}$ in the preimage. Then $\sigma_{\lambda,\P}(g)$ is contained in some open neighbourhood in $W$, i.e.~in the coset $\sigma_{\lambda,\P}(g) \cdot (\Sym X_i)_{(\Psi)}$ of the pointwise stabiliser of some finite set $\Psi \subseteq X_i$. Next, let $\Phi$ be the finite set $\{c\in\P \mid \lambda_i(c)\in\Psi\}$. Note that for all $h \in (\Aut{\Delta})_{(\Phi)}$ we have that
	\[\sigma_\lambda(gh,\P) = \sigma_\lambda(g,h\acts\P) \cdot \sigma_\lambda(h,\P)\]
	where $h$ stabilises the panel $\P$ and $\sigma_\lambda(h,\P)$ fixes the set $\Psi$. It thus follows that
	\[\sigma_{\lambda,\P}\big(g\cdot(\Aut{\Delta})_{(\Phi)}\big) \subseteq \sigma_{\lambda,\P}(g) \cdot (\Sym{X_i})_{(\Psi)} \subseteq W,\]
	hence $g$ is contained in the open neighbourhood
	\[g\cdot(\Aut{\Delta})_{(\Phi)} \subseteq \sigma_{\lambda,\P}^{-1}(W).\]
	Since $g$ was arbitrary, the preimage $\sigma_{\lambda,\P}^{-1}(W)$ is open.
\end{proof}
\begin{proposition}
	\label{prop:closed}
	The following are equivalent:
	\begin{enumerate}
		\item for each $i\in I$, the local group $F_i$ is closed in $\Sym(X_i)$;
		\item $\U_\Delta^\lambda(\F)$ is closed in $\Aut(\Delta)$.
	\end{enumerate}
\end{proposition}
\begin{proof}
	For the implication {(i) $\Rightarrow$ (ii)}, we give a different proof from the one in \cite{demedts2018} that also works in the case where $\Delta$ is not locally finite, following the ideas of \cite[Lemma 14]{smith2017}. For each $i\in I$ and $\P\in\Res_i(\Delta)$, we consider the map
	\[\sigma_{\lambda,\P} \colon \Aut(\Delta)\to\Sym(X_i) \colon g\mapsto\sigma_\lambda(g,\P),\]
	which is continuous by \cref{lem:local-cont}. Our claim then follows from the observation that
	\[\U_\Delta^\lambda(\F) = \bigcap_{i\in I}\:\bigcap_{\P\in\Res_i(\Delta)}\sigma_{\lambda,\P}^{-1}(F_i).\]
	We prove the implication {(ii) $\Rightarrow$ (i)} by contraposition, so assume that $F_i$ is not closed for some $i\in I$. Recall that the permutation topology agrees with the topology of pointwise convergence, and let $(f_n)_{n\in\mathbb N}$ be a sequence of elements of $F_i$ such that $\lim f_n=f$ exists but $f\notin F_i$. Let $\P$ be an $i$-residue. For every $n\in\mathbb N$, let $g_n$ be the automorphism of $\Delta$ given by extending $f_n$ as in \cref{lem:extendlocaluniversal}. As long as we use the same permutations $\varphi_i(x,y)$ in the proof of \cref{lem:extendlocaluniversal} for different $n$, the sequence $(g_n)_{n\in\mathbb N}$ converges to an automorphism $g$ that stabilises $\P$ with $\sigma_\lambda(g,\P)=f$, hence $g=\lim g_n\notin\U_\Delta^\lambda(\F)$. This shows that $\U_\Delta^\lambda(\F)$ is not closed in $\Aut(\Delta)$.
\end{proof}

\begin{proposition}[local compactness]
\label{prop:locallycompact}
	Assume that all $F_i$ are closed in $\Sym(X_i)$. Then the following are equivalent:
	\begin{enumerate}
		\item every point stabiliser in every $F_i$ is compact;
		\item every chamber stabiliser in $\U(\F)$ is compact;
		\item $\U(\F)$ is a locally compact topological group.
	\end{enumerate}
\end{proposition}
\begin{proof}
	Since all $F_i$ are closed, $\U(\F)$ is closed by \cref{prop:closed}. Its point stabilisers are open subgroups and hence closed as well.
	
	For the	implication {(i) $\Rightarrow$ (ii)}, consider an arbitrary chamber $c$. We need to show that the stabiliser $\U_c$ has only finite orbits. Take a second chamber $d$ and let
	\[ c=c_0 \sim c_1 \sim c_2 \sim \dots \sim c_{k-1} \sim c_k=d \]
	be a minimal gallery from $c$ to $d$. Then $\big|{\U_c}\acts d\big|$ is at most
	\begin{align*}
		\big|{\U_{c_0}}\acts c_1\big| \cdot \big|{\U_{\{c_0,c_1\}}}\acts c_2\big| \dotsm \big|{\U_{\{c_0,\dots,c_{k-1}\}}}\acts c_k\big|
		 \leq \big|{\U_{c_0}}\acts c_1\big| \cdot \big|{\U_{c_1}}\acts c_2\big| \dotsm \big|{\U_{c_{k-1}}}\acts c_k\big|.
	\end{align*}
	Since the local action of every point stabiliser in the latter product is permutationally isomorphic to a point stabiliser in some $F_i$, and the $F_i$ are compact, all these factors are finite. 

	\smallskip
	The implication {(ii) $\Rightarrow$ (iii)} is obvious.

	\smallskip
	Finally we show the implication {(iii) $\Rightarrow$ (i)}. By local compactness of the universal group we can find a finite set of chambers $\Phi$ whose pointwise stabiliser in $\U(\F)$ is compact. Thus, every orbit of $\U_{(\Phi)}$ is finite.
	
	Pick any index $i$ and any $i$-colour $x\in X_i$. Then by \cref{lem:finiteinwing}, there exists an $i$-panel $\P$ such that $\proj_\P(\Phi)$ equals a single chamber $c$ with $\lambda_i(c)=x$. From \cref{lem:extendlocaluniversal} it follows that $X_i$ cannot contain orbits of $(F_i)_x$ of infinite size, since otherwise the corresponding chambers in $\P$ would lie in an infinite orbit of $\U_{(\Phi)}$. Thus, every point stabiliser in every local group $F_i$ is compact.
\end{proof}

\begin{proposition}[compact generation of panel stabilisers]
	\label{prop:stabcompgen}
	Assume that $\U(\F)$ is locally compact. Then, for each $j\in I$, the following are equivalent:
	\begin{enumerate}
		\item the local group $F_j$ is compactly generated;
		\item the setwise stabiliser in $\U(\F)$ of a $j$-panel is compactly generated.
	\end{enumerate}
\end{proposition}
\begin{proof}
	Let $\P$ be a $j$-panel and $c$ a chamber in $\P$. Consider the isomorphism~$\varphi$ mapping $\restrict{\U}{\P}$ to $F_j$ from \cref{localpermiso}. Clearly $\varphi$ maps the subgroup $\restrict{\U_c}{\P}$ to the point stabiliser $(F_j)_x$ where $x=\lambda_j(c)$; note that $(F_j)_x$ is compact by assumption.
	
	We will first show that (ii) $\Rightarrow$ (i).
	So assume that $\U_{\{\P\}}=\langle Q\rangle$ with $Q$ a compact subset of $\U$. Then the union of cosets $\bigcup_{q\in Q} q\U_c$ is an open cover of $Q$; by compactness, there is some finite subcover with representatives $\{q_1,\dots,q_k\}\subseteq Q$. It follows that $\U_{\{\P\}} = \langle q_1,\dots,q_k,\U_c\rangle$. Then
	\[F_j = \langle\varphi(\restrict{q_1}{\P}),\dots,\varphi(\restrict{q_k}{\P}),(F_j)_x\rangle,\]
	which shows that $F_j$ is compactly generated.

	\smallskip
	Next, we show that (i) $\Rightarrow$ (ii).
	So assume that $F_j=\langle Q\rangle$ with $Q$ a compact subset of $F_j$. Again, we find, by compactness, that $F_j=\langle q_1,\dots,q_k,(F_j)_x\rangle$ with $\{q_1,\dots,q_k\}\subseteq Q$. For each $1\leq i\leq k$, let $h_i$ be any extension of the element $\varphi^{-1}(q_i)\in\restrict{\U}{\P}$ defined on the whole building, i.e., let $h_i\in\U_{\{\P\}}$ be such that $\varphi(\restrict{h_i}{\P})=q_i$. Now consider the compactly generated subgroup
	\[H = \langle h_1,\dots,h_k,\U_c\rangle \leq \U_{\{\P\}}.\]
	We claim that, in fact, $H = \U_{\{\P\}}$.
	Indeed, for each $g\in\U_{\{\P\}}$, there is some $h\in H$ with $g\acts c=h\acts c$. But then $h^{-1}g\in\U_c$ so that $g\in H$ as claimed.
\end{proof}

As pointed out by the referee, one can strenghten \cref{prop:stabcompgen} to the following explicit structural description of a panel stabiliser.
\begin{proposition}
\label{referee1}
	Assume $\U(\F)$ is locally compact. Let $\P$ be a $j$-panel. Then $\U_{(\P)}$ is a compact normal subgroup of $\U_{\{\P\}}$, and moreover, the quotient $\U_{\{\P\}} / \U_{(\P)}$ and $F_j$ are isomorphic as topological groups.
\end{proposition}
\begin{proof}
	Consider the map $\sigma_{\lambda,\P}$ from \cref{lem:local-cont}. This map $\sigma$ restricts to a continuous epimorphism $\U_{\{\P\}} \to F_j$ with compact kernel $\U_{(\P)}$.
\end{proof}

\begin{proposition}[discreteness]
	\label{prop:discrete}
	The following are equivalent:
	\begin{enumerate}
		\item Each $F_i$ acts freely on $X_i$;
		\item $\U(\F)$ acts freely on $\Delta$;
		\item $\U(\F)$ is a discrete topological group.
	\end{enumerate}
\end{proposition}
\begin{proof}
	For the implication {(i) $\Rightarrow$ (ii)}, suppose that all $F_i$ act freely and that $g\in\U(\F)$ fixes some chamber $c$. Then for each $i\in I$, the local action of $g$ at the $i$-panel containing~$c$ is a permutation in $F_i$ with a fixed point, and hence is the identity. It follows that every panel containing $c$ is fixed by $g$. By connectedness of the building, $g$ is the identity.
	
	\smallskip
	The implication {(ii) $\Rightarrow$ (iii)} is obvious.
	
	\smallskip
	We prove the final implication {(iii) $\Rightarrow$ (i)} by contraposition. Suppose some $F_j$ does not act freely on $X_j$ and let $x\in X_j$ be a colour such that the stabiliser $(F_j)_x$ is non-trivial. Consider any finite set of chambers $\Phi$. By \cref{lem:finiteinwing}, there exists a $j$-panel $\P$ such that $\proj_\P(\Phi)=\{c\}$, with $\lambda_j(c)=x$. Let $f\in (F_j)_x$ be non-trivial. Then by \cref{lem:extendlocaluniversal}, this permutation~$f$ extends to a non-trivial element of the universal group that stabilises $\Phi$. Hence, no pointwise stabiliser of a finite set of chambers is trivial.
\end{proof}

\begin{definition}
	For a permutation group $F\leq\Sym(X)$, denote by $F^+$ the subgroup generated by all point stabilisers in $F$. Similarly, denote by $\U(\F)^+$ the subgroup of the universal group $\U(\F)$ generated by all chamber stabilisers.
\end{definition}

\begin{lemma}
	\label{lem:pluslocal}
	The local actions of $\U(\F)^+$ on the $i$-panels all lie in $F_i^+$.
\end{lemma}
\begin{proof}
	Let $c$ be a chamber in $\Delta$ and let $g\in\U_c$. It suffices to show that the local action $\sigma(g,\mathcal P)$ is a permutation in $F_i^+$ (since then the same holds for products of stabilisers). A proof follows by induction on the distance
	\[\dist(c,\P)=\min\{\dist(c,d)\mid d\in\P\}.\]
	If $\dist(c,\P)=0$, then $c\in\P$ and $\sigma(g,\P)\in F_i$ fixes $\lambda_i(c)$. If $\dist(c,\P)=n+1$, let $\P'$ be an adjacent $i$-panel at distance $n$ from $c$. Let $d'$ and $d$ be $j$-adjacent chambers in $\P$ and $\P'$ respectively, for some $j\neq i$. We know that $\lambda_i(d')=\lambda_i(d)$ and that
	\[\sigma(g,\P') \acts \lambda_i(d') = \lambda_i(g\acts d') = \lambda_i(g\acts d) = \sigma(g,\P) \acts \lambda_i(d),\]
	so that $\sigma(g,\P) = \sigma(g,\P')\cdot f$ for some permutation $f\in F_i$ stabilising the colour $\lambda_i(d)$. Since $\sigma(g,\P')\in F_i^+$ by the induction hypothesis, we find that indeed $\sigma(g,\P) \in F_i^+$.
\end{proof}

\section{Compact generation}\label{se:compgen}

In this section, we discuss some conditions on the local groups for a universal group~$\U$ to be compactly generated. Again, we make no assumptions regarding finiteness or transitivity of the local groups, but we do assume $\U$ to be locally compact. We are grateful to Pierre-Emmanual Caprace for discussing this problem with us and helping establish the rank 2 case, thereby suggesting a way to generalise to higher rank.

We will need the following result, taken from \cite{moller03}.
\begin{lemma}[good generating sets]
	\label{lem:goodgeneratingset}
	Let $G$ be a compactly generated totally disconnected locally compact group. Let $V$ be a compact open subgroup of $G$. Then there exists a finite set $T=\{\tau_1,\dots,\tau_n\}$, closed under inverses, such that $G=\langle T\rangle\cdot V$. The group $V$ together with the set $T$ is called a \emph{good generating set} of $G$.
\end{lemma}
\begin{proof}
	See \cite[Lemma 2]{moller03}.
\end{proof}

In the following lemma, the fact that $\Delta$ is right-angled is crucial; the result does not at all hold for, say, general spherical buildings.
\begin{lemma}
	\label{lem:nicesubbuildings}
	Let $\lambda$ be a colouring of a right-angled building $\Delta$ using colour sets $X_i$. For each $i\in I$, let $Y_i\subseteq X_i$ be a subset of the $i$-colours such that $|Y_i|\geq 2$. Let $c_0\in\Delta$ be any chamber, and let $\Gamma$ be the set of chambers of $\Delta$ that are connected to $c_0$ by a gallery that only takes colours in the restricted sets $Y_i$. Then $\Gamma$ is a semiregular subbuilding with the same type as $\Delta$ and with parameters $q_i=|Y_i|$.
\end{lemma}
\begin{proof}
	By \cite[Proposition 7.18]{weiss2003} or \cite[Theorem 4.66]{abramenko-brown}, it suffices to show that $\Gamma$ is convex. So consider an arbitrary gallery $\gamma$ in $\Gamma$ and a minimal gallery $\gamma'$ in $\Delta$ such that $\gamma$ and $\gamma'$ are homotopic; we will show that also $\gamma'$ is contained in $\Gamma$.
	By \cite{ronan}, we can reduce $\gamma$ to $\gamma'$ by elementary operations of two types, namely elementary contractions and elementary homotopies. We will show that applying such an operation on a gallery in $\Gamma$ results in a gallery that is again contained in $\Gamma$.
	\begin{enumerate}
		\item (\emph{elementary contractions}) If a gallery $\gamma$ in $\Gamma$ of type $w\cdot ii\cdot w'$ is contracted to a gallery $\gamma'$ of type $w\cdot w'$ or $w\cdot i\cdot w'$, then the chambers of $\gamma'$ form a subset of the chambers of $\gamma$, so $\gamma'$ is again a gallery in $\Gamma$;
		\item (\emph{elementary homotopies}) If a gallery in $\Gamma$ of type $ij$ with $m_{ij}=2$ is transformed into a gallery of type $ji$, then this new gallery is again contained in $\Gamma$. Indeed, whenever $c\sim_i d\sim_j c'$ with $c,d,c'\in\Gamma$ and $c\sim_j d'\sim_i c'$ with $d' \in \Delta$, we have $\lambda_i(d')=\lambda_i(c)$ and $\lambda_j(d')=\lambda_j(c')$, so that $d'\in\Gamma$.
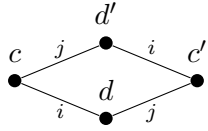
			\[\begin{tikzpicture}[x=12mm,y=5mm]
				\node[myvertex,label={above:$c$}] (C) at (-1,0) {};
				\node[myvertex,label={above:$c'$}] (C') at (1,0) {};
				\node[myvertex,label={above:$d$}] (D) at (0,-1) {};
				\node[myvertex,label={above:$d'$}] (D') at (0,1) {};
				\draw[font=\scriptsize\strut,inner sep=0pt]
					(C) to node[below] {$i$} (D)
						to node[below] {$j$} (C')
					(C) to node[above] {$j$} (D')
						to node[above] {$i$} (C');
			\end{tikzpicture}\]
	\end{enumerate}
	This shows that $\Gamma$ is convex.
\end{proof}

\begin{lemma}
	\label{lem:convexhullfinite}
	Any finite set of chambers in a semiregular right-angled building $\Delta$ is contained in a finite convex set of chambers.
\end{lemma}
\begin{proof}
	First, we prove that any $n$ chambers $\{c_1,\dots,c_n\}$ are contained in a semiregular \emph{locally finite} subbuilding of $\Delta$. For $n=1$, this is obvious. For $n\geq 2$, we will make use of a legal colouring $\lambda$ of $\Delta$. Join~$c_1$ to every chamber in $\{c_2,\dots,c_n\}$ by an arbitrary minimal gallery. Let $Y_i\subseteq X_i$ be the subset of all $i$-colours that occur as the $i$-colour of some chamber on one of these newly added minimal galleries, and note that $Y_i$ is finite (for every $i\in I$). Let $\Gamma$ be the set of chambers of $\Delta$ that are connected to $c_1$ by a gallery that only takes colours in the sets $(Y_i)_{i\in I}$.
	
	By \cref{lem:nicesubbuildings}, $\Gamma$ is a locally finite subbuilding of $\Delta$ containing the chambers $\{c_1,\dots,c_n\}$. In $\Gamma$, these chambers can be enclosed by a \emph{finite} ball $B$ of finite radius. Moreover, as $B$ is convex in $\Gamma$ and $\Gamma$ is convex in $\Delta$, it follows that $B$ is a finite convex set in $\Delta$ containing $\{c_1,\dots,c_n\}$.
\end{proof}

\begin{definition}[minimal actions]
	Let $G$ act on a chamber system by automorphisms. We say that $G$ acts \emph{minimally} if there are no non-trivial $G$-invariant convex subsystems.
\end{definition}

Universal groups act minimally on their associated buildings, except in some very degenerate cases. This follows from the following lemma, together with \cref{cor:cobounded}. We thank the referee for proposing a generalisation of our original statement.

\begin{lemma}
\label{referee2}
	Let $\Delta$ be a right-angled building such that the diagram does not have isolated nodes. Let $G\leq\Aut(\Delta)$, and assume that the action of $G$ on $\Delta$ is cobounded. Then the action of $G$ on $\Delta$ is minimal.
\end{lemma}
\begin{proof}
	Suppose that $\U$ stabilises some non-empty convex subsystem $\Gamma$ of $\Delta$. Let $c\sim_i d$ be such that $c\in\Gamma$ but $d\notin\Gamma$. Since the action is cobounded, there exists some constant $R$ such that every chamber lies within distance $R$ of some chamber in the orbit $G\acts c$.
	
	Using the absence of isolated nodes, take $j\in I$ such that $m_{ij}=\infty$ and let $\mathcal A$ be an apartment of type $\{i,j\}$ containing $c$ and $d$. Note that $\mathcal A$ is a bi-infinite ray. Let $d'$ be the chamber on $\mathcal A$ such that $\dist(c,d')=\dist(d,d')+1=R+1$. Next, find $c'\in G\acts c$ such that $\dist(c',d')\leq R$. We now have a minimal gallery $\gamma_1$ from $c$ to $d'$ of length $R+1$ and a minimal gallery $\gamma_2$ from $d'$ to $c'$ of length at most $R$.
	
	We can convert the concatenation $\gamma_1\gamma_2$ to a minimal gallery using only elementary homotopies, contractions, and operations of the form $ii\mapsto i$ on its type. In doing so, we obtain a minimal gallery $\gamma$ from $c$ to $c'$ whose initial segment $c\sim_i d$ remains unchanged (by comparing the lengths of $\gamma_1$ and $\gamma_2$).
	
	Thus, $\gamma$ is a minimal gallery joining $c\in\Gamma$ to $c'\in G\acts c\subseteq G\acts\Gamma=\Gamma$ and containing $d$, so that $d\in\Gamma$ by convexity. This contradiction shows that $\Gamma=\Delta$.
\end{proof}
\begin{corollary}
	\label{lem:uminimal}
	Let $\Delta$ be a right-angled building such that the diagram does not have isolated nodes. Then $\U_\Delta^\lambda(\F)$ acts minimally on $\Delta$.

\end{corollary}

\begin{theorem}
	\label{thm:compgen1}
	Assume that $\U(\F)$ is closed, locally compact, and compactly generated. Then for each $i\in I$, the local group $F_i$ has finitely many orbits.
\end{theorem}
\begin{proof}
	Pick any chamber $c\in\Delta$. Consider the chamber stabiliser $\U_c$ which is a compact open subgroup by \cref{prop:locallycompact}. Let $T=\{\tau_1,\dots,\tau_n\}$ be a finite set as in \cref{lem:goodgeneratingset}, so that $T$ together with $\U_c$ is a good generating set. By \cref{lem:convexhullfinite}, we can find a finite convex set $B \subseteq \Delta$ containing $\{c,\tau_1\acts c,\dots,\tau_n\acts c\}$.
	
	First, we claim that the set $\U\acts B$ is connected. Let $g \in \U$ and $d \in B$ be arbitrary; we will show that $g \acts d$ is connected to $c$ by some gallery contained in $\U\acts B$.  Write $g\in\U=\langle T\rangle\cdot\U_c$ as $g=t_k\dotsm t_1\cdot s$ (with $t_1,\dots,t_k\in T$, $s\in\U_c$); we will use induction on~$k$. For $k=0$, we simply have that $g\in\U_c$ so that $g\acts d$ is connected to $g\acts c=c$ by a gallery in $g\acts B$. For $k\geq 1$, let $g=t_k\cdot g'$. By the induction hypothesis, $g'\acts d$ is connected to $c$ by a gallery in $\U\acts B$. It follows that $g\acts d$ is connected to $t_k\acts c$ by a gallery in $\U\acts B$, and we can concatenate this gallery with one from $t_k\acts c$ to $c$ contained in $B$. We conclude that the set $\U\acts B$ is indeed connected.
	
	Next, we claim that each orbit of $F_i$ has a ``representative chamber'' in $B$, or more precisely, we claim that for each $i\in I$, the map
	\[\tilde\lambda_i\colon B\to X_i/F_i\colon b\mapsto F_i\acts\lambda_i(b)\]
	is surjective. In order to show this, suppose that some orbit $Y$ of $F_j$ is not in the image of $\tilde\lambda_j$ for some $j\in I$. Call a chamber \emph{neglected} if its $j$-colour is contained in the orbit $Y$. By assumption, no chamber in $B$ is neglected, and \cref{lem:orbits} implies that no chamber in $\U\acts B$ is neglected. Let $\Gamma$ be the set of all chambers of $\Delta$ that are connected to $\U \acts B$ by galleries that do not pass through neglected chambers.
	Then by \cref{lem:nicesubbuildings}, $\Gamma$ is a semiregular subbuilding of $\Delta$ and by construction, $\Gamma$ is $\U$-invariant. This contradicts the minimality of the action of $\U$ (see \cref{lem:uminimal}).
	
	We conclude that for each $i\in I$, we have a surjective map $\tilde\lambda_i$ from a finite set~$B$ onto the set of orbits of $i$-colours. Hence the local groups $F_i$ have finitely many orbits.
\end{proof}

\begin{theorem}
	\label{thm:compgen2}
	Assume that $\U(\F)$ is closed and locally compact. Moreover, assume that for each $i\in I$, the group $F_i$ is compactly generated and has finitely many orbits. Then the universal group $\U(\F)$ is compactly generated.
\end{theorem}
\begin{proof}
	\setcounter{stepcounter}{0}
	We will construct an explicit compact generating set $Q$. For each $i \in I$, we choose a transversal $\mathcal T_i$ for the action of $F_i$ on the colours $X_i$. By assumption, all $\mathcal T_i$ are finite. Let $c\in\Delta$ be a chamber such that $\lambda_i(c)\in\mathcal T_i$ for all $i\in I$. Define the ball
 	\[\widetilde B = \{d\in\Delta \mid \text{$\dist(c,d) \leq |I|$}\}\]
	and the finite subset
	\[B = \{d\in\widetilde B \mid \text{$\lambda_i(d)\in\mathcal T_i$ for all $i\in I$}\}.\]
	Note that $\U\acts B=\Delta$ by \cref{lem:orbits} because each possible colour combination from colours in the sets $\mathcal T_i$ occurs inside $B$. Next, we define
	\[\widetilde D = \{d\in\Delta \mid \text{$\dist(c,d) = |I|+1$ and $d$ is adjacent to a chamber in $B$}\}\]
	and the finite subset
	\[D = \{d\in\widetilde D \mid \text{$\lambda_i(d)\in\mathcal T_i$ for all $i\in I$}\}.\]
	For each pair $(b,d)\in B\times (B\cup D)$ of equally coloured chambers, pick an element $t_{(b,d)}$ in $\U$ that maps $b$ to $d$. Let $T$ be the finite set of all such elements; note that $D\subseteq T\acts B$.
	
	By assumption, all local groups are compactly generated, hence by \cref{prop:stabcompgen}, all setwise panel stabilisers are compactly generated as well. For each panel $\P$, let $S_\P$ be a compact generating set for $\U_{\{\P\}}$; then define the compact set
	\[S=\bigcup  \{S_\P \mid \text{$\P$ is a panel containing a chamber in $B$}\}.\]

	Finally, let $Q={\U_c}\cup S\cup T$ and observe that $Q$ is compact. We claim that $\langle Q\rangle=\U$. In order to prove this claim, we proceed in four steps.

	\begin{proofstep}
		For each chamber $d\in\widetilde B$, we have $d\in\langle S\rangle\acts B$.
	\end{proofstep}
	We show this by induction on $\dist(c,d)$. If $d=c$ there is nothing to show. Otherwise, let $d'$ and $i$ be such that $d'\sim_i d$ and $\dist(c,d')=\dist(c,d)-1$. Since $d'\in\widetilde B$, we can write $d'=h\acts b'$ for some $h\in\langle S\rangle$ and $b'\in B$. Now let $\P$ be the $i$-panel that contains~$b'$ and let $b$ be the unique chamber of $\P$ such that $\lambda_i(b)$ is the representative of the orbit of $\lambda_i(h^{-1}\acts d)$. Then there is some $s\in \U_{\{\P\}}$ such that $s\acts b = h^{-1}\acts d$, and we conclude that indeed $d=hs\acts b\in\langle S\rangle\acts B$.

	\begin{proofstep}
		For each chamber $d\in\widetilde D$, we have $d\in\langle S,T\rangle\acts B$.
	\end{proofstep}
	By definition of $\widetilde D$, the chamber $d$ is $i$-adjacent to some chamber in $B$, for some $i$. Let $\P$ be the $i$-panel containing $d$. Let $d'$ be the unique chamber in $\P$ such that $\lambda_i(d')$ is the representative of the orbit of $\lambda_i(d)$. Then there are two possible cases. If $d'\in B$, we can immediately set $b=d'$ and $t=1$. Otherwise $d'\in D$, and we can find some $b\in B$ and $t\in T$ sending $b$ to $d'$. Moreover, there exists some $s\in \U_{\{\P\}}$ sending $d'$ to $d$, and we conclude that $d=st\acts b\in\langle S,T\rangle\acts B$.

	\begin{proofstep}
		For each chamber $d\in\Delta$, we have $d\in\langle S,T\rangle\acts B$.
	\end{proofstep}
	We again use induction on $\dist(c,d)$. If $d=c$, there is nothing to show. Else, let $d'$ be a chamber such that $d'\sim_i d$ and $\dist(c,d')=\dist(c,d)-1$. By induction, $d'=h\acts b$ for some $h\in\langle S,T\rangle$ and some $b\in B$. Notice that $\dist(d, h \acts c) \leq |I|+1$. There are three cases to consider.
	\begin{enumerate}
		\item If $d\in h\acts B$, then there is nothing to prove.

		\item If $d\notin h\acts B$ but $d\in h\acts\widetilde B$, then $h^{-1}\acts d\in\widetilde B$. From Step 1, we know that $h^{-1}\acts d\in\langle S\rangle\acts B$, so that $d\in h\langle S\rangle\acts B$.

		\item If $d\notin h\acts\widetilde B$, then $\dist(d,h\acts c)=|I|+1$. Since $h^{-1}\acts d$ is adjacent to $b\in B$, we have $h^{-1}\acts d\in\widetilde D$. From Step 2, we know that $h^{-1}\acts d\in\langle S,T\rangle\acts B$, so that $d\in h\langle S,T\rangle\acts B$.
	\end{enumerate}
	In each case, we find that $d\in\langle S,T\rangle\acts B$ and we conclude that $\langle S,T\rangle\acts B=\Delta$.
	
	\begin{proofstep}
		$\langle Q\rangle = \langle S,T,\U_c\rangle = \U$.
	\end{proofstep}
	Let $g\in\U$ be arbitrary. From Step 3, we know that $g\acts c=h\acts b$ for some $h\in\langle S,T\rangle$ and some $b\in B$. Note that the chambers $b$ and $c$ both lie in the same orbit of $\U$ and in~$B$, so by \cref{lem:orbits}, they have the same colours. By definition of $T$, there is an element $t\in T$ mapping $c$ to $b$. Putting everything together, we see that $ht\acts c = h \acts b =g\acts c$, from which it follows that $g\in ht\cdot\U_c$. In particular, $g\in\langle Q\rangle$.
\end{proof}

Combining \cref{lem:orbits}, \cref{thm:compgen1}, and \cref{thm:compgen2}, we obtain as a corollary:
\begin{corollary}
	Assume that $\U(\F)$ is closed and locally compact. Moreover assume that $F_i$ is compactly generated for each $i\in I$. Then the following are equivalent:
	\begin{itemize}
		\item $\U(\F)$ is compactly generated;
		\item $\U(\F)$ has finitely many orbits on $\Delta$;
		\item $F_i$ has finitely many orbits for each $i$.
	\end{itemize}
\end{corollary}

The natural question is now whether the converse to \cref{thm:compgen2} holds as well. We established half of the converse in \cref{thm:compgen1}, but we have not yet found a complete and general proof.
\begin{conjecture}
	\label{theconjecture}
	Assume that each local group $F_i$ is closed, that $\U(\F)$ is locally compact, and that $\U(\F)$ is compactly generated. Does it then follow that each $F_i$ is compactly generated?
\end{conjecture}
If true, this would give a precise characterisation when a closed and locally compact universal group is compactly generated: $\U(\F)$ is compactly generated if, and only if, the local groups $F_i$ are compactly generated and have finitely many orbits.
	
As a motivation we present some (affirmative) results for specific special cases. First we assume some additional information about the subgroups $F_i^+$ of the local groups $F_i$. In the following proofs, we use the fact that every locally compact group is a directed union of its compactly generated open subgroups; see \cite[Proposition 2.C.3 (1)]{cornulier-delaharpe}.
\begin{theorem}
	\label{thm:compgen3}
	Assume that $\U(\F)$ is closed, locally compact, and compactly generated. Moreover, assume that every $F_i^+$ is compactly generated. Then the local groups $F_i$ are compactly generated.
\end{theorem}
\begin{proof}
	For ease of notation, we fix one $F_i=F$ and show that $F$ is compactly generated. Choose any chamber $c\in\Delta$, let $x=\lambda_i(c)$, and let $H_{(0)}$ be the stabiliser of $x$ in $F$. Now write $F$ as a directed union
	\newcommand{\dircup}{\mathop{\overrightarrow{\bigcup}}\displaylimits}
	\[F = \dircup_{k\in\mathbb I} H_{(k)}\]
	of compactly generated open subgroups, where $0$ is the least element of the index set $\mathbb I$. For each $k\in\mathbb I$, let $K_{(k)}$ be the subgroup of $\U$ generated by the chamber stabiliser $\U_c$ and by $\U(F_1,\ldots,F_{i-1},H_{(k)},F_{i+1},\ldots,F_n)$. Then $\U$ is the directed union
	\[\U = \dircup_{k\in\mathbb I} K_{(k)}.\]
	By assumption $\U$ is compactly generated, say $\U = \langle Q\rangle$ with $Q$ compact. The collection of subgroups $\{K_{(k)} \mid k\in\mathbb I\}$ defines an open cover of $Q$. By compactness, if follows that $Q\subseteq K_{(k)}$ for some $k$, i.e.~$\U = K_{(k)}$.
	
	Now consider the local actions of $\U$ at the $i$-panels. By definition and by \cref{localpermiso}, the local actions are given by the group $F$. On the other hand the local actions of $K_{(k)}$ are generated by $H_{(k)}$ and the local actions of $\U_c$. Hence $F=\langle H_{(k)},F^+\rangle$ (by \cref{lem:pluslocal}). As $H_{(k)}$ is compactly generated by construction and $F^+$ by assumption, our conclusion follows.
\end{proof}

\begin{remark}
	It seems difficult to get rid of the assumption that the groups $F_i^+$ are compactly generated. As Derek Holt kindly pointed out (\cite{mathoverflowholt}), this property does not follow automatically from \cref{thm:compgen1} and the other assumptions in \cref{thm:compgen3}. More precisely, if $F\leq\Sym(X)$ has finitely many orbits and all point stabilisers have finite orbits, then $F$ need not be generated by finitely many point stabilisers---not even in the more restrictive setting where $F=F^+$.
	
	Consider the following counterexample. Let $F=A\rtimes\langle t\rangle$, where $A$ is an infinite abelian group of odd finite exponent, $t^2=1$, and $t^{-1}at=a^{-1}$ for all $a\in A$. Let $F$ act by (left) translation on the set $X$ of (left) cosets of $\langle t\rangle$ in $F$. Then $F$ acts transitively on $X$, every point stabiliser $F_{a\langle t\rangle} = \{1,a^2t\}$ has order $2$, and finally $F = F^+$, but $F$ is not generated by a finite number of point stabilisers.
	
	Endowed with the permutation topology $F$ is totally disconnected, locally compact, and has finitely many orbits, but $F=F^+$ is not compactly generated. The question remains whether or not $F$ can occur as a local group of a compactly generated universal group.
\end{remark}

Furthermore, when $\Delta$ is of rank $2$ (i.e.~a tree), we do not need assumptions on $F_i^+$. One can show the following proposition using Bass--Serre theory.
\begin{proposition}
\label{prop:compgenfortrees}
	Let $G$ be a compactly generated totally disconnected locally compact group. Suppose that $G$ acts discretely on a tree $\mathcal T$ without edge inversions, such that the quotient graph $G\backslash\mathcal T$ is finite and the edge stabilisers are compact open subgroups. Then the vertex stabilizers are compactly generated.
\end{proposition}
\begin{proof}
	This argument is due to Colin Reid (\cite{mathoverflowreid}). Let $v_1,\ldots,v_k$ be representatives for the vertices in the quotient graph. Let $E$ be a set of representatives for the edges in the quotient, and let $E'\subseteq E$ be a subset of edges representing a spanning tree in $G\backslash\mathcal T$.
	
	Using Bass--Serre theory, we can write $G$ in the form
	\[G \cong \frac{\strut G_{v_1} \ast \dots \ast G_{v_k} \ast F(E)}%
				{\llangle\strut \overline{e}\cdot\alpha_e(g)\cdot e = \alpha_{\overline{e}}(g) \enspace\text{for $g \in G_e$}, 			\quad e\cdot\overline{e} \enspace\text{for $e \in E$},
								\quad e \enspace\text{for $e \in E'$}\rrangle},\]
	where $\alpha_e\colon G_e \rightarrow G_{o(e)}$ is the natural embedding of an edge group into the vertex group of the origin of the edge.
	
	Now fix a vertex $v=v_1$. The subgroup $\left\langle \alpha_e(G_e) \mid e\in E, o(e)=v\right\rangle \leq G_v$ is open and compactly generated; denote this group by $H_{(0)}$. Now write $G_v$ as a directed union
	\newcommand{\dircup}{\mathop{\overrightarrow{\bigcup}}\displaylimits}
	\[G_v = \dircup_{i\in\mathbb I} H_{(i)}\]
	of compactly generated open subgroups, where $0$ is the least element of the index set $\mathbb I$. For each $i\in\mathbb I$, let $K_{(i)} = \left\langle H_{(i)}, G_{v_2},\ldots, G_{v_k}, F(E)\right\rangle$. Then $G$ is the directed union
	\[G = \dircup_{i\in\mathbb I} K_{(i)}.\]
	By assumption $G$ is compactly generated, say $G=\langle Q\rangle$ with $Q$ compact. The collection of (compactly generated and open) subgroups $\{K_{(i)} \mid i\in\mathbb I\}$ defines an open cover of $Q$. By compactness, if follows that $Q\subseteq K_{(i)}$ for some $i\in\mathbb I$, i.e.~$G = K_{(i)}$. More explicitly, every $g\in G_v$ is a product of elements in $H_{(i)}\cup G_{v_2}\cup{\cdots}\cup G_{v_k}\cup F(E)$. By the normal form theorem for graphs of groups (\cite{higgings76}), we then conclude that $g\in H_{(i)}$, i.e.~that $G_v = H_{(i)}$ is compactly generated.
\end{proof}
\begin{corollary}
	Assume that $\U(F_1,F_2)$ is closed, locally compact, and compactly generated. Then also the local groups $F_1$ and $F_2$ are compactly generated.
\end{corollary}
\begin{proof}
	This follows immediately from \cref{prop:compgenfortrees} using $\mathcal T=\Delta$ and $G=\U(F_1,F_2)$. Note that the quotient graph $G\backslash\mathcal T$ is finite by \cref{thm:compgen1}.
\end{proof}

\begin{remark}
	The similarity between the proofs of \cref{thm:compgen3} and \cref{prop:compgenfortrees} is of course striking. Finding a suitable generalisation of the necessary Bass--Serre theory and normal forms (e.g.~using the theory of complexes of groups by Martin Bridson and Andr\'e Haefliger  \cite{bridson13}) might be a way of settling \cref{theconjecture}.
\end{remark}


\section{Simplicity}
\label{se:simp}

In \cite[Theorem 23]{smith2017}, Simon Smith proved a simplicity result for universal groups in rank 2 (i.e.~for trees), in terms of the permutation properties of the two local groups. We extend his result to right-angled buildings of higher rank, where also the combinatorial properties of the diagram of $\Delta$ will be of importance.

In this section, let $\U=\U_\Delta(\F)$ be the universal group and let $\U^+$ be the subgroup of $\U$ generated by all chamber stabilisers\footnote{Note that this notation differs from \cite{caprace14} where the notation $\Aut(\Delta)$ is meant to also include non-type-preserving automorphisms, and $\Aut(\Delta)^+$ denotes the subgroup of type-preserving automorphisms of $\Delta$. In this paper we only consider type-preserving automorphisms.}. $\U^+$ is a normal subgroup and hence a major obstruction for the universal group to be simple. The primary goal in this section is to characterise when $\U^+=\U$; we can then follow the proof of \cite[Theorem 3.25]{demedts2018} to characterise when $\U$ is (abstractly) simple.

First we need some results about \emph{implosions} of right-angled buildings: we show that we have a controlled way to ``collapse'' a coloured building by collapsing the colour sets.
\begin{proposition}
	\label{th:implosion}
	Let $\Delta$ be a semiregular right-angled building over $I$ and let $\lambda$ be a legal colouring of $\Delta$ using the colour sets $X_i$ (for $i\in I$). For every $i\in I$, consider an equivalence relation $\equiv_i$ on $X_i$, and let $I'=\{i\in I\mid\text{${\equiv_i}$ is \emph{not} the universal relation}\}$.
	
	 Define the new semiregular right-angled building $\Delta'$ over $I'$, the diagram of which is induced by the diagram of $\Delta$, with parameters $q'_i = \lvert X_i/{\equiv_i}\rvert$ (for every $i\in I'$). Equip $\Delta'$ with a legal colouring $\lambda'$ using the quotient set $X_i/{\equiv_i}$ as the set of $i$-colours.

	Then there exists a surjective map $\tau\colon\Delta\to\Delta'$ such that $\lambda'_i(\tau(c)) = [\lambda_i(c)]_i$ for all chambers $c\in\Delta$ and such that $\dist_{\Delta'}(\tau(c_1),\tau(c_2)) \leq \dist_\Delta(c_1,c_2)$ for all $c_1,c_2\in\Delta$.
\end{proposition}
\begin{proof}
	Let $c_0\in\Delta$, $c'_0\in\Delta'$ be two chambers such that $\lambda'_i(c'_0) = [\lambda_i(c_0)]_i$ for all $i\in I'$. We construct $\tau$ by induction on the distance from $c_0$, settling the induction base by declaring $\tau(c_0)=c'_0$. For $c\in\Delta$ such that $\dist(c_0,c)=n+1$, let $d$ be such that $\dist(c_0,d)=n$ and $d\sim_i c$. If $\lambda_i(c)\equiv_i\lambda_i(d)$ (in particular, if $i\notin I'$), then we set $\tau(c)=\tau(d)$. Otherwise we set $\tau(c)$ to be the unique chamber in $\Delta'$ $i$-adjacent to $\tau(d)$ such that $\lambda'_i(\tau(c))=[\lambda_i(c)]_i$.
	
	Since $\tau(c)$ a priori depends on the choice of $d$, we need to show that $\tau$ is well-defined. In order to do so, suppose that both $d_1$ and $d_2$ satisfy $\dist(c_0,d_1)=n=\dist(c_0,d_2)$ and $d_1\sim_i c\sim_j d_2$. By \cref{lem:closingsquares} (i), there exists a chamber $e$ such that $\dist(c_0,e)=n-1$ and $d_1\sim_j e\sim_i d_2$. Since $\lambda_i(e)=\lambda_i(d_1)$, $\lambda_i(d_2)=\lambda_i(c)$, $\lambda_j(e)=\lambda_j(d_2)$, $\lambda_j(d_1)=\lambda_j(c)$, the images of the paths $e\sim d_1\sim c$ and $e\sim d_2\sim c$ end up in the same chamber in $\Delta'$, so $\tau(c)$ is indeed well-defined.
	
	This extends $\tau$ to the whole of $\Delta$, and by construction, we have $\lambda'_i(\tau(c)) = [\lambda_i(c)]_i$ for all $c\in\Delta$.
	
	It is not hard to see that $\tau$ is surjective, since any gallery $\gamma'$ in $\Delta'$ can be ``lifted'' to a gallery $\gamma$ in $\Delta$ such that $\tau(\gamma)=\gamma'$. Explicitly, let $\gamma'$ be a gallery
	\[d'_0 \sim_{i_1} d'_1 \sim_{i_2} d'_2 \sim_{i_3} {\cdots} \sim_{i_n} d'_n\]
	in $\Delta'$. For every $1\leq k\leq n$, let $x_k$ be a representative of the equivalence class $\lambda'_{i_k}(d'_{k})$. Let $d_0\in\Delta$ be any chamber such that $\tau(d_0)=d'_0$, and, for every $1\leq k\leq n$, let $d_k\in\Delta$ be the unique chamber $i_k$-adjacent to $d_{k-1}$ such that the $i_k$-colour of $d_k$ equals $x_k$. Note that the $i_k$-colours of $d_{k-1}$ and $d_k$ cannot be $i_k$-equivalent, since then $d'_{k-1}$ and $d'_k$ would be the same chamber in $\Delta'$; hence $d_k$ is well-defined. If we then let $\gamma$ be the gallery
	\[d_0 \sim_{i_1} d_1 \sim_{i_2} d_2 \sim_{i_3} {\cdots} \sim_{i_n} d_n\]
	in $\Delta$, we clearly have $\tau(\gamma)=\gamma'$. In particular we see that $\tau$ is surjective (using $d'_0=c'_0$).

	It remains to show that $\tau$ is nonexpansive, i.e.~that $\dist_{\Delta'}(\tau(c_1),\tau(c_2)) \leq \dist_\Delta(c_1,c_2)$ for all $c_1,c_2\in\Delta$. It suffices to show this for adjacent chambers, so assume that $c_1\sim_j c_2$. We use induction on the distance to $c_0$. If $\dist(c_0,c_1)\neq\dist(c_0,c_2)$, then this is clear (by the very definition of $\tau$). Suppose that $\dist(c_0,c_1)=\dist(c_0,c_2)=n+1$ and let $d_1$ satisfy $\dist(c_0,d_1)=n$ and $d_1\sim_i c_1$. There are two possibilities. First, if $i=j$, then also $d_1\sim_i c_2$ in $\Delta$ and either $\tau(c_1)=\tau(c_2)$ or $\tau(c_1)\sim_j\tau(c_2)$ in $\Delta'$. Second, if $i\neq j$, then by \cref{lem:closingsquares} (ii) there exists a chamber~$d_2$ such that $\dist(c_0,d_2)=n$ and $d_1\sim_j d_2\sim_i c_2$. Since $\lambda_i(c_1)=\lambda_i(c_2)$, $\lambda_i(d_1)=\lambda_i(d_2)$, $\lambda_j(c_1)=\lambda_j(d_1)$, $\lambda_j(c_2)=\lambda_j(d_2)$, and since either $\tau(d_1)=\tau(d_2)$ or $\tau(d_1)\sim_j\tau(d_2)$, we see that either $\tau(c_1)=\tau(c_2)$ or $\tau(c_1)\sim_j\tau(c_2)$ in $\Delta'$ as well.
\end{proof}
\begin{definition}
	\label{def:implosion}
	 We call $(\tau,\Delta')$ an \emph{implosion} of $\Delta$ (with respect to the relations $\equiv_i$).
\end{definition}

The following proposition shows why implosions are useful for understanding $\U^+$.
\begin{proposition}
	\label{th:implosion2}
	For every $i\in I$, define $\equiv_i$ by declaring $i$-colours to be equivalent if and only if they are contained in the same orbit of $F_i$. Let $(\tau,\Delta')$ be an implosion of $\Delta$ with respect to $\equiv_i$. Then the group $\U_\Delta(\F)^+$ stabilises the fibres of $\tau$.
\end{proposition}
\begin{proof}
	Let $g\in\U(\F)$ stabilise some chamber $c\in\Delta$ and consider any chamber $d\in\Delta$ together with its image $g\acts d$. We show that $\tau(d)=\tau(g\acts d)$ by induction on $\dist(c,d)$. The base case with $c=d$ is trivial, since $g\acts c=c$. If $\dist(c,d)=n+1$, then let $e\in\Delta$ be such that $\dist(c,e)=n$ and $e\sim_i d$ (for some $i\in I$). We distinguish two cases.
	\begin{enumerate}[label={(\alph*)}]
		\item If $\lambda_i(e)\equiv_i\lambda_i(d)$, i.e.~$e$ and $d$ are harmonious, then $g\acts e$ and $g\acts d$ are harmonious as well. As $\tau(d)=\tau(e)$ by construction, $\tau(e)=\tau(g\acts e)$ by the induction hypothesis, and $\tau(g\acts e)=\tau(g\acts d)$ again by construction, we may conclude that $\tau(d)=\tau(g\acts d)$.
		\item If $\lambda_i(e)\not\equiv_i\lambda_i(d)$, then by construction,
			\begin{itemize}
				\item $\tau(d)$ is the unique chamber $i$-adjacent to $\tau(e)$ in $\Delta'$ such that $\lambda'_i(\tau(d))$ equals the orbit $F_i\acts\lambda_i(d)$, and
				\item $\tau(g\acts d)$ is the unique chamber $i$-adjacent to $\tau(g\acts e)$ in $\Delta'$ such that $\lambda'_i(\tau(g\acts d))$ equals the orbit of $F_i\acts\lambda_i(g\acts d)$.
			\end{itemize}
			Now since $\tau(e)=\tau(g\acts e)$ by the induction hypothesis, and since $\lambda_i(d)$ and $\lambda_i(g\acts d)$ clearly lie in the same orbit of $F_i$, we conclude that also $\tau(d)=\tau(g\acts d)$.
	\end{enumerate}
	By induction, the fibres of $\tau$ are stabilised by $\U_c$ and hence also by $\U^+$.
\end{proof}

\begin{remark}
	\label{rem:fibres}
	Note that the relations $\equiv_i$ uniquely determine the building $\Delta'$, but not the map $\tau$ from \cref{th:implosion}. However, as soon as we decide on two ``compatible'' base chambers $c_0\in\Delta$ and $c'_0\in\Delta'$ such that $[\lambda_i(c_0)] = \lambda_i(c'_0)$, there is a \emph{unique} implosion map $\tau$ mapping $c_0$ to $c'_0$. For this choice of $(c_0,c'_0)$, we then have a canonical morphism
	\[\psi\colon \Aut(\Delta)\to\Aut(\Delta')\]
	defined by $\psi(g)\acts\tau(d) = \tau(g\acts d)$ for every $g\in\Aut(\Delta)$ and $d\in\Delta$. Observe that the kernel of $\psi$ contains precisely those automorphisms $g$ that stabilise all fibres of $\tau$. In other words, \cref{th:implosion2} can be rephrased as the statement that $\U_\Delta(\F)^+ \leq \ker(\psi)$.
	
	On the other hand, note that $\U_\Delta(\F)^+ \neq \ker(\psi)$ in general. Indeed, as an extreme case, let the local groups be \emph{regular} permutation groups, i.e., transitive and with trivial point stabilisers. Then by \cref{lem:pluslocal} we have that $\U_\Delta(\F)^+$ is trivial, whereas an implosion map collapses $\Delta$ to a single chamber, i.e.~in this case the kernel of $\psi$ is the full automorphism group of $\Delta$.
\end{remark}

\begin{definition}
	We call a subset $V\subseteq I$ a \emph{vertex cover} of the diagram if for all $i,j\in I$ such that $m_{ij}=\infty$, we have $i\in V$ or $j\in V$ (or both).
	
	Note that this corresponds to the classical, graph-theoretical notion of a vertex cover if one interprets the diagram as a graph in the natural way, by declaring vertices $i,j\in I$ to be adjacent if and only if $m_{ij}=\infty$.
\end{definition}
\begin{theorem}
	\label{thm:genbychamstab}
	The following are equivalent.
	\begin{enumerate}
		\item $\U_\Delta(\F)^+=\U_\Delta(\F)$;
		\item the local groups $F_i$ are generated by point stabilisers for all $i\in I$, and transitive for all $i$ in some vertex cover of the diagram of $\Delta$.
	\end{enumerate}
\end{theorem}
\begin{proof}
	First we show the implication (i) $\Rightarrow$ (ii). If $\U^+=\U$, then every local action of $\U$ on an $i$-panel is a permutation in $F_i^+$ by \cref{lem:pluslocal}. At the same time, \cref{lem:extendlocaluniversal} shows that every permutation in $F_i$ can be realised as a local action at an $i$-panel. Hence indeed $F_i=F_i^+$ for all $i\in I$.
	
	\def\colorJ{{\vartriangleleft}}
	\def\colorJJ{{\blacktriangleleft}}	
	\def\colorK{{\vartriangleright}}
	\def\colorKK{{\blacktriangleright}}	
	For the transitivity result, we use an implosion map $\tau\colon\Delta\to\Delta'$ and \cref{th:implosion2}. Suppose by means of contradiction that the indices of the transitive local groups do \emph{not} define a vertex cover of the diagram of $\Delta$, i.e.~that there are intransitive local groups $F_j$ and $F_k$ such that $m_{jk}=\infty$. Then $\Delta'$ is not spherical.
	
	Let $\colorJ,\colorJJ\in X_j$ be two colours in different orbits of~$F_j$ and $\colorK,\colorKK\in X_k$ two colours in different orbits of~$F_k$. Consider an apartment $\mathcal A$ of a $\{j,k\}$-residue with colours only in $\{\colorJ,\colorJJ\}$ and $\{\colorK,\colorKK\}$, so that $\mathcal A$ looks as follows.
	
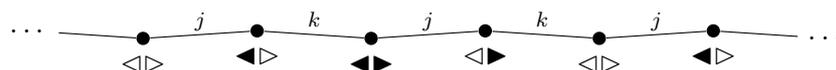
\[\begin{tikzpicture}[x=15mm,y=1mm]
		\node[myvertex,label={below:$\colorJ\colorK$}] (C0) at (0,0) {};
		\node[myvertex,label={below:$\colorJJ\colorK$}] (C1) at (1,1) {};
		\node[myvertex,label={below:$\colorJJ\colorKK$}] (C2) at (2,0) {};
		\node[myvertex,label={below:$\colorJ\colorKK$}] (C3) at (3,1) {};
		\node[myvertex,label={below:$\colorJ\colorK$}] (C4) at (4,0) {};
		\node[myvertex,label={below:$\colorJJ\colorK$}] (C5) at (5,1) {};
		\node (T0) at (-1,1) {\dots};
		\node (T5) at (6,0) {\dots};
		\draw[font=\scriptsize\strut,inner sep=0pt]
			(T0) -- (C0) to node[above] {$j$} (C1)
				to node[above] {$k$} (C2)
				to node[above] {$j$} (C3)
				to node[above] {$k$} (C4)
				to node[above] {$j$} (C5) -- (T5);
	\end{tikzpicture}\]
	Now let $c,d\in\mathcal A$ be two chambers such that $\dist(c,d)=4$. Then $c$ and $d$ are harmonious, hence they lie in the same orbit of $\U$ (\cref{lem:orbits}). At the same time, $\tau$ maps adjacent chambers in $\mathcal A$ to \emph{distinct} adjacent chambers in a $\{j,k\}$-residue in $\Delta'$. Hence $\tau(c)\neq\tau(d)$, and from \cref{th:implosion2} it follows that $c$ and $d$ do not lie in the same orbit of $\U^+$. Thus $\U\neq\U^+$ (in fact, this argument shows that $\U^+$ is not even of finite index in $\U$).
	
	\smallskip
	For the converse (ii) $\Rightarrow$ (i), it is sufficient to show that $\U$ and $\U^+$ have identical orbits on $\Delta$ (since their chamber stabilisers agree). Let $c,d$ be two chambers in the same orbit of $\U$, i.e.~harmonious chambers. We will find an automorphism in $\U^+$ taking $c$ to $d$.
	
	First assume that $c$ and $d$ are contained in some \emph{spherical} residue $\R$ of type $J\subset I$. Recall that
	\[\U_\R(\{F_j\}_{j\in J}) \cong \prod_{j\in J} F_j.\]
	Now, for every $j\in J$, there exists some permutation $f_j\in F_j$ such that $\lambda_j(c)=f_j\acts\lambda_j(d)$. By assumption, $f_j\in F_j=F_j^+$ can be written as a product of permutations each of which stabilises some $j$-colour. Extend these permutations from the $j$-panels containing $c$ to building automorphisms of $\Delta$ using \cref{lem:extendlocaluniversal}, and take their product over all $j\in J$. This gives us an automorphism in~$\U^+$ that stabilises $\R$ and that maps $c$ to a chamber with $J$-colours equal to those of $d$, i.e., to $d$ itself.
	
	Now consider the general case and let $\gamma$ be a minimal gallery from $c$ to $d$. We use induction on the length $n$ of $\gamma$; the case $n=0$ is trivial and the case $n=1$ follows from the previous paragraph. For arbitrary $n > 1$, we have to consider two cases, depending on the number of pairs of consecutive harmonious chambers on $\gamma$.
	\begin{enumerate}[label={(\alph*)}]
		\item If there are no such pairs, let $J_\gamma$ be the set of all $i\in I$ occurring in the type of~$\gamma$. Then for each $i\in J_\gamma$ the local group $F_i$ has at least two orbits (i.e.~$F_i$ is intransitive). By the vertex cover assumption on the local groups, any two elements $i,j\in J_\gamma$ commute in the Coxeter group of $\Delta$. Hence the residue of type $J_\gamma$ containing $\gamma$ is spherical, and the conclusion follows from the paragraph above.
		\item If there is at least one such pair $(x,y)$, then using the induction base $n=1$ we can find an element $g\in\U^+$ such that $g\acts x = y$. Now let $d' = g^{-1}\acts d$ and note that a minimal gallery from $c$ to $d'$ has length at most $n-1$. By the induction hypothesis, there is an element $h\in\U^+$ such that $h\acts c = d'$. Then $gh\in\U^+$ satisfies $gh\acts c = d$.
	\end{enumerate}
	This finishes the proof.
\end{proof}

\begin{theorem}[simplicity]\label{thm:simplicity}
	Let $\Delta$ be a thick, semiregular, right-angled building of irreducible type $I$, with rank $|I|\geq 2$ and parameters $(q_i)_{i\in I}$. Let $\lambda$ be a legal colouring of $\Delta$. For each $i\in I$, let $F_i\leq\Sym(X_i)$. Then the universal group $\U_\Delta^\lambda(\F)$ is (abstractly) simple if and only if the local groups $F_i$ are generated by point stabilisers for all $i\in I$ and transitive for all $i$ in some vertex cover of the diagram of $\Delta$.
\end{theorem}
\begin{proof}
	By \cref{thm:genbychamstab}, both conditions on $F_i$ are necessary for $\U$ to coincide with~$\U^+$. For the converse we can follow the proof of \cite[Theorem 3.25]{demedts2018} \emph{ad verbatim}, since the proof of this theorem (or the previous lemmata) does not require the building to be locally finite; the proof essentially demonstrates that every non-trivial normal subgroup of $\U$ contains $\U^+$ as a subgroup.
\end{proof}

\section{Primitivity}
\label{se:prim}

In this final section, we characterise when a universal group acts primitively on the set of building residues of some fixed type. Our main result, \cref{thm:prim}, has been proved by Simon Smith in the rank 2 case in \cite[Theorem 26 (ii)]{smith2017}.

Moreover, as noted by the referee, the main theorem of \cite[Theorem 4.29]{demedts2019} implies that stabilisers of maximal residues are maximal subgroups of $\Aut(\Delta)$, so that the full automorphism group $\Aut(\Delta)$ acts primitively on the set of $J$-residues for every maximal $J\subsetneq I$. We will recover this particular case in \cref{thm:prim}.

We will need Higman's renowned primitivity criterion (\cite{higman67}). For convenience we recall some definitions and state the criterion.
\begin{definition}[orbitals]
	Let $G$ be a group acting on a set $X$. This action induces an action of $G$ on $X\times X$ by $g\acts(x,y)=(g\acts x,g\acts y)$, whose orbits are called \emph{orbitals} of $G$.
\end{definition}
\begin{definition}[orbital graphs]
	The \emph{orbital graph} with respect to an orbital $\mathcal O$ is the (directed) graph with vertex set $X$ and edge set $\mathcal O\subseteq X\times X$.
\end{definition}
Suppose that $G$ is transitive. There is one trivial or \emph{diagonal} orbital $\{(x,x) \mid x\in X\}$, whose orbital graph simply has a loop at each vertex. No other orbital graph has loops.
\begin{theorem}[Higman]
	\label{thm:higman}
	Suppose that $G$ acts transitively on some set $X$. Then $G$ acts primitively on $X$ if and only if every non-diagonal orbital graph is (weakly) connected\footnote{As opposed to \emph{strongly connected}, a directed graph is \emph{weakly connected} or simply \emph{connected} if its underlying undirected graph is connected.}.
\end{theorem}
\begin{proof}
	See \cite[(1.12)]{higman67}.
\end{proof}

We will also need the following folklore result.
\begin{lemma}
\label{lem:primitivenonregular}
	Let $G$ be a primitive, non-regular permutation group on $X$. Consider two distinct elements $x,y\in X$. Then there exists a permutation $g\in G$ such that $g\acts x=x$ but $g\acts y\neq y$.
\end{lemma}
\begin{proof}
	Suppose (by means of contradiction) that $G_x\subseteq G_y$. Since $G$ is primitive, $G_x$ is a maximal subgroup, hence $G_x=G_y$. Let $g\in G$ be such that $g\acts x=y$; note that $g\notin G_x$ but that $g$ normalises $G_x$. We can hence write
	\[G_x < N_G(G_x) \norleq G.\]
	Again, as $G_x$ is maximal, it follows that $G_x\norleq G$. But then all point stabilisers are equal to $G_x$, so that $G_x$ fixes all points of $X$, i.e., $G$ is regular---a contradiction.
\end{proof}

\begin{lemma}
\label{lem:welldefinedicolours}
	Let $I=J\sqcup\{k\}$ be a partition of the index set. Let $\R, \R'$ be two distinct $J$-residues, $c,d$ chambers in $\R$, and $c',d'$ chambers in $\R'$, such that $c\sim_k c'$ and $d\sim_k d'$.
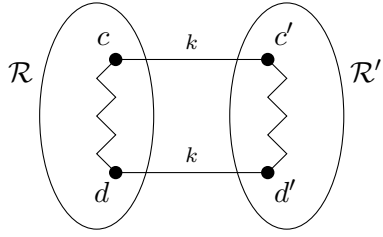
	\[\begin{tikzpicture}
		\DrawlabelledEllipse{(-1.25,0)}{.75}{1.5}{0}{$\R$}{150};
		\DrawlabelledEllipse{(1.25,0)}{.75}{1.5}{0}{$\R'$}{30};
		\draw (-1,.75) node[myvertex,label={150:$c$}] {} to node[above] {\scriptsize $k$} (1,.75) node[myvertex,label={30:$c'$}] {}
		 \foreach\i in {1,2,3} {--++ (.25,-.25) --++ (-.25,-.25)} node[myvertex,label={330:$d'$}] {} to node[above] {\scriptsize $k$} (-1,-.75) node[myvertex,label={210:$d$}] {}
		 \foreach\i in {1,2,3} {--++ (-.25,.25) --++ (.25,.25)};
	\end{tikzpicture}\]
	Finally, let $i\in J\setminus \{k\}^\perp$. Then $\lambda_i(c) = \lambda_i(c') = \lambda_i(d) = \lambda_i(d')$.
\end{lemma}
\begin{proof}
	Let $\P$ be the $k$-panel containing $c$ and $c'$, and let $\P'$ be the $k$-panel containing $d$ and $d'$. From convexity of residues it follows that $c=\proj_\P(d)$ and $c'=\proj_\P(d')$. Hence, by \cref{prop:parallelism}~(i), the panels $\P$ and $\P'$ are parallel. By \cref{prop:parallelism}~(ii), $c$, $d$, $c'$, $d'$ are contained in a common residue of type $k\cup k^\perp$. Finally, since $i\notin k\cup k^\perp$, the $i$-colours of these chambers are identical.
\end{proof}

\begin{theorem}[primitivity]
	\label{thm:prim}
	Let $J\subsetneq I$. The universal group $\U(\F)$ acts primitively on the set $\Res_J(\Delta)$ of $J$-residues if and only if the following conditions hold:
	\begin{enumerate}
		\item $|I\setminus J|=1$, so that $I=J\sqcup\{k\}$ for some $k\in I$,
		\item $F_k$ is primitive and non-regular,
		\item $F_i$ is transitive for all $i\in I$ such that $m_{ik}=\infty$.
	\end{enumerate}
\end{theorem}
\begin{proof}
	\setcounter{stepcounter}{0}
	We proceed in five steps.
	
	\begin{proofstep}
		Condition {\upshape(i)} is necessary.
	\end{proofstep}
	
	If $|I\setminus J|\geq 2$, then $J$-residues are contained in (strictly bigger) maximal residues. The maximal residues of a fixed type containing $J$ partition the set of $J$-residues and determine a block system of imprimitivity for the action of $\U$.
	
	\medskip

	In the following steps we may assume that $I=J\sqcup\{k\}$. Define the graph $\Gamma$ with the~set $\Res_J(\Delta)$ of $J$-residues as its vertex set, in which two $J$-residues are adjacent if and only if they contain two $k$-adjacent chambers. Note that $\U$ acts in a natural way on $\Gamma$ by graph automorphisms. We will simply call $J$-residues \textit{adjacent} when there are adjacent in $\Gamma$.
	
	\begin{proofstep}
		Assume that $F_k$ is not primitive. Then $\U$ acts imprimitively on $\Res_J(\Delta)$.		
	\end{proofstep}

	Suppose that $F_k$ acts imprimitively on the $k$-colours $X_k$ and let $\approx$ be a non-trivial equivalence relation on $X_k$ invariant under $F_k$. Let $\P$ be any $k$-panel and lift $\approx$ to an equivalence relation on the chambers of $\P$ in the obvious way, i.e., let $c_1\approx c_2$ in~$\P$ if and only if $\lambda_k(c_1)\approx\lambda_k(c_2)$ in $X_k$.

	If $q_k=2$ (i.e., in the thin case), $F_k\leq\Sym(2)$ being imprimitive means that $F_k$ is trivial. Hence the action of $\U$ on $\Res_J(\Delta)$ is not even transitive, let alone primitive.
	
	If $q_k\geq 3$ (i.e., in the thick case), let $c_1$, $c_2$ and $c_3$ be three chambers in $\P$ such that $\lambda_k(c_1)\approx\lambda_k(c_2)\not\approx\lambda_k(c_3)$. Let $\R_1$, $\R_2$ and $\R_3$ be the $J$-residues containing $c_1$, $c_2$ and $c_3$, respectively.
	
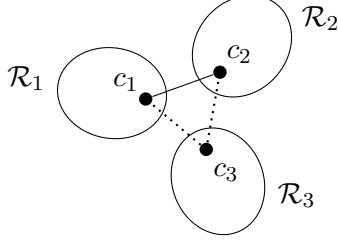
\[\begin{tikzpicture}[scale=.6]
		\node[myvertex,label={170:$c_1$}] (d) at (170:1) {};
			\DrawlabelledEllipse{(170:1.75)}{1.2}{1}{170}{$\R_1$}{170};;
		\node[myvertex,label={20:$c_2$}] (c) at (50:1) {};
			\DrawlabelledEllipse{(50:1.75)}{1.2}{1}{50}{$\R_2$}{20};
		\node[myvertex,label={350:$c_3$}] (e) at (290:1) {};
			\DrawlabelledEllipse{(290:1.75)}{1.2}{1}{290}{$\R_3$}{350};;
		\draw (c) -- (d) (e) edge[dotted,thick] (d) edge[dotted,thick] (c);
	\end{tikzpicture}\]
	Consider the orbital graph w.r.t.~the orbital $\U\acts(\R_1,\R_2)$, and note that this graph is a subgraph of $\Gamma$. We claim that $\R_3$ does not lie in the same connected component as $\R_1$ and $\R_2$ in this orbital graph, so that the imprimitivity of $\U$ on $\Res_J(\Delta)$ follows from Higman's theorem (\cref{thm:higman}). The easiest way to see this is by using the (well-defined) $k$-colours of the residues: the action of $\U$ preserves whether or not the $k$-colours of adjacent $J$-residues are equivalent (by \cref{localpermiso} and invariance of $\approx$). It follows that the connected component of $\R_1$ and $\R_2$ in the orbital graph contains only residues with $k$-colour equivalent to $\lambda_k(c_1)$ and $\lambda_k(c_2)$, and in particular does not contain $\R_3$, as claimed.
	
	\begin{proofstep}
		Assume that $F_k$ is regular. Then $\U$ acts imprimitively on $\Res_J(\Delta)$.
	\end{proofstep}
	
	Suppose that $F_k$ acts regularly on the $k$-colours $X_k$ and consider any colour $y\in X_k$, any residue $\R_0\in\Res_J(\Delta)$ and any element $g\in\U$. Let $z\in X_k$ be the colour of $g\acts\R_0$. By regularity, there is a \emph{unique} element $f\in F_k$ such that $f\acts y=z$. We show that $\lambda_k(g\acts \R)=f\acts\lambda_k(\R)$ for each $J$-residue $\R$, using induction on the distance $n$ to $\R_0$.

	The base case $n=0$ is trivial: by definition, $\lambda_k(g\acts\R_0)=z=f\acts y=f\acts\lambda_k(\R_0)$.
	
	Consider a $J$-residue $\R$ at distance $n+1$ from $\R_0$ and let $\R'$ be a neighbour of~$\R$ in~$\Gamma$ at distance $n$ from $\R_0$. Let $x=\lambda_k(\R)$ and $x'=\lambda_k(\R')$, and notice that by the induction hypothesis, $\lambda_k(g\acts \R') = f\acts x'$. Since $g$ is a building automorphism, $g\acts\R$ is again adjacent to $g\acts\R'$ in $\Gamma$. Let $\P$ be the $k$-panel containing the chambers in $\R$ and $\R'$ that are $k$-adjacent. The local action of $g$ at $\P$ is a permutation of $F_k$ and maps $x'$ to $f\acts x'$, so by regularity this local action has to be equal to $f$. Hence $g\acts\R$ has $k$-colour $f\acts x$, or in other words $\lambda_k(g\acts\R)=f\acts\lambda_k(\R)$.
	
	Our claim now follows by induction, and implies that the sets
	\[B(x) = \{\R\in\Res_J(\Delta) \mid \lambda_k(\R)=x\}\]
	of identically coloured residues are blocks of imprimitivity, as $g\acts B(x) = B(f\acts x)$.

	\begin{proofstep}
		Condition {\upshape(iii)} is necessary.
	\end{proofstep}

	Suppose that $F_i$ acts intransitively on $X_i$ for some $i\in I$ such that $m_{ik}=\infty$. From \cref{lem:welldefinedicolours} it follows that the $k$-panels induce well-defined $i$-colours on the edges of~$\Gamma$. Consider an arbitrary $J$-residue $\R_1$ and pick two $i$-adjacent chambers $c\sim_i d$ in $\R_1$ such that $\lambda_i(c)$ and $\lambda_i(d)$ lie in distinct orbits of $F_i$. Let $\R_2$ and $\R_3$ be two $J$-residues containing a chamber that is $k$-adjacent to $c$ and $d$, respectively.
	
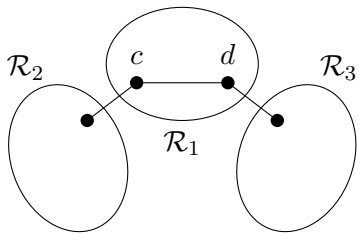
\[\begin{tikzpicture}
		\DrawlabelledEllipse{(0,1.25)}{.75}{1}{90}{$\R_1$}{270};
		\DrawlabelledEllipse{(-1.5,0)}{.75}{1}{20}{$\R_2$}{115};
		\DrawlabelledEllipse{(1.5,0)}{.75}{1}{340}{$\R_3$}{65};
		\draw (-1.25,.5) node[myvertex] {}
			-- (-.6,1) node[myvertex,label={90:$c$}] {}
			-- (.6,1) node[myvertex,label={90:$d$}] {}
			-- (1.25,.5) node[myvertex] {};
	\end{tikzpicture}\]
	Consider the orbital graph w.r.t.~the orbital $\U\acts(\R_1,\R_2)$ (again a subgraph of $\Gamma$). We claim that $\R_3$ does not lie in the same connected component as $\R_1$ and $\R_2$ in this orbital graph. It suffices to note that the action of $\U$ on $\Gamma$ preserves the $F_i$-orbit of an edge's $i$-colour. Hence the $i$-colours in every connected component of the orbital graph are contained in a single orbit of $F_i$, and in particular is $\R_3$ not contained in the component of $\R_1$ and $\R_2$ by construction.

	\begin{proofstep}
		Assume conditions {\upshape(i), (ii), (iii)} fulfilled. Then $\U$ acts primitively on $\Res_J(\Delta)$.
	\end{proofstep}
	
	Let $\approx$ be a non-trivial $\U$-invariant equivalence relation on the set $\Res_J(\Delta)$ of $J$-residues; we will show that this relation is universal. Consider two equivalent residues $\R_0\approx\R$, consider a shortest path from $\R_0$ to $\R$ in $\Gamma$ and let $\R'$ be the $J$-residue adjacent to $\R$ on this shortest path. Let $c\in\R$ and $c'\in\R'$ be $k$-adjacent chambers and let $\P$ be the $k$-panel containing $c$ and $c'$.

	As $F_k$ is primitive and non-regular, there exists a permutation $f\in F_k$ fixing $\lambda_k(c')$ but not $\lambda_k(c)$ by \cref{lem:primitivenonregular}. By \cref{lem:extendlocaluniversal}, $f$ extends to an automorphism $g\in\U$ stabilising $\P$, fixing $c'$ but not $c$, and fixing all chambers $d$ such that $\proj_\P(d)=c'$. In particular, $g$ fixes $c_0$. Moreover $g\acts c$ is $k$-adjacent to $g\acts c'=c'$.
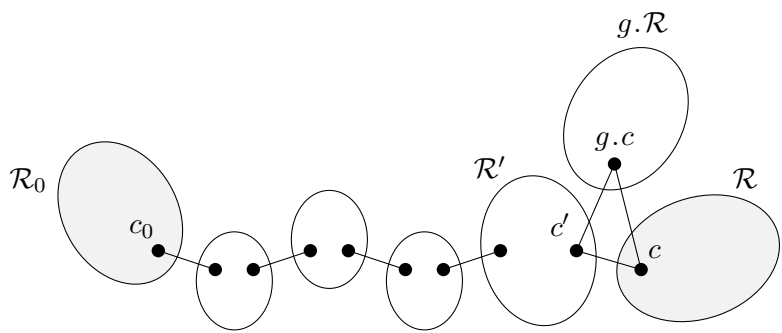
	\[\begin{tikzpicture}
		\DrawlabelledEllipse[fill=black!5]{(0,1)}{.75}{1}{30}{$\R_0$}{160};
		\DrawEllipse{(1.5,.1)}{.5}{.65}{0};
		\DrawEllipse{(2.75,.65)}{.5}{.65}{0};
		\DrawEllipse{(4,.1)}{.5}{.65}{0};
		\DrawlabelledEllipse{(5.5,.5)}{.75}{1}{10}{$\R'$}{120};
		\DrawlabelledEllipse[fill=black!5]{(7.6,.4)}{1.1}{.8}{20}{$\R$}{60};
		\DrawlabelledEllipse{(6.65,2.25)}{.75}{1}{150}{$g\acts\R$}{80};
		\draw (.5,.5) node[myvertex,label={120:$c_0$}] {} -- (1.25,.25) node[myvertex] {}
			(1.75,.25) node[myvertex] {} -- (2.5,.5) node[myvertex] {}
			(3,.5) node[myvertex] {} -- (3.75,.25) node[myvertex] {}
			(4.25,.25) node[myvertex] {} -- (5,.5) node[myvertex] {};
		\draw (6,.5) node[myvertex,label={150:$c'$}] {}
			-- (6.85,.25) node[myvertex,label={3:$c$}] {}
			-- (6.5,1.65) node[myvertex,label={90:$g\acts c$}] {} -- cycle;
	\end{tikzpicture}\]
	From $\U$-invariance it follows that $\R\approx\R_0=g\acts\R_0\approx g\acts\R$, so we have constructed two equivalent $J$-residues \emph{adjacent in $\Gamma$}.
	
	We now claim that all $J$-residues containing a chamber in $\P$ are equivalent. In order to see this, consider the induced equivalence relation on the $k$-colours~$X_k$, where we define two colours to be equivalent if the $J$-residues of the chambers in~$\P$ with those $k$-colours are equivalent. This gives a non-trivial equivalence relation on $X_k$ invariant under $F_k$ (by \cref{lem:extendlocaluniversal}), which is universal by primitivity of $F_k$. Hence our claim follows, and in particular, we find that $\R\approx\R'\approx g\acts\R$.

	Next, we claim that all $J$-residues that are adjacent to $\R$ are equivalent to $\R$. Let $\R''$ be such a residue adjacent to $\R$, and let $d\in\R$ be $k$-adjacent to a chamber in $\R''$. Consider a minimal gallery $\gamma$ from $c$ to $d$ in $\R$; we will show that $\R \approx \R''$ by means of induction on the gallery length $n$ of $\gamma$. The case $n=0$ is precisely the conclusion of the previous paragraph. For a gallery $\gamma$ of length $n+1$,
	\[ c=c_0 \sim c_1 \sim c_2 \sim \dots \sim c_n=d' \sim c_{n+1}=d , \]
	let $i\in I$ be such that $d' \sim_i d$. There are two options.
	
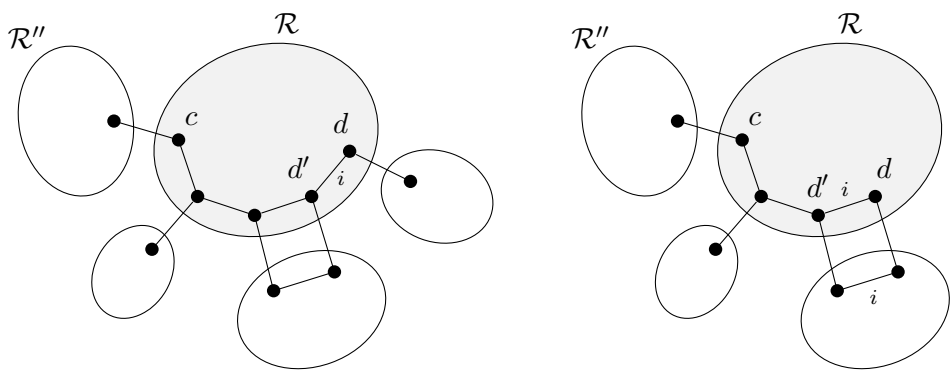
\[\begin{tikzpicture}
		\DrawlabelledEllipse{(5.5,.5)}{.75}{1}{10}{$\R''$}{120};
		\DrawlabelledEllipse[fill=black!5]{(8,.25)}{1.5}{1.25}{20}{$\R$}{80};
		\DrawEllipse{(6.25,-1.5)}{.65}{.5}{60};
		\DrawEllipse{(8.6,-2)}{1}{.75}{20};
		\DrawEllipse{(10.25,-.5)}{.75}{.6}{160};
		\draw (6,.5) node[myvertex] {} -- (6.85,.25) node[myvertex,label={30:$c$}] {}
			-- (7.1,-.5) node[myvertex] {} edge node[myvertex,pos=1] {} (6.5,-1.2)
			-- (7.85,-.75) node[myvertex] {} edge (8.6,-.5) 
			-- (8.1,-1.75) node[myvertex] {} -- (8.9,-1.5) node[myvertex] {} -- (8.6,-.5) node[myvertex,label={100:$d'$}] {}
 			-- node[right,pos=.4] {\scriptsize $i$} (9.1,.1) node[myvertex,label={100:$d$}] {} edge node[myvertex,pos=1] {} (9.9,-.3);
	\end{tikzpicture}
	\qquad
	\begin{tikzpicture}
		\DrawlabelledEllipse{(5.5,.5)}{.75}{1}{10}{$\R''$}{120};
		\DrawlabelledEllipse[fill=black!5]{(8,.25)}{1.5}{1.25}{20}{$\R$}{80};
		\DrawEllipse{(6.25,-1.5)}{.65}{.5}{60};
		\DrawEllipse{(8.6,-2)}{1}{.75}{20};
		\draw (6,.5) node[myvertex] {} -- (6.85,.25) node[myvertex,label={30:$c$}] {}
			-- (7.1,-.5) node[myvertex] {} edge node[myvertex,pos=1] {} (6.5,-1.2)
			-- (7.85,-.75) node[myvertex,label={90:$d'$}] {} edge node[above,pos=.4] {\scriptsize $i$} (8.6,-.5) 
			-- (8.1,-1.75) node[myvertex] {} -- node[below,pos=.6] {\scriptsize $i$} (8.9,-1.5) node[myvertex] {} -- (8.6,-.5) node[myvertex,label={80:$d$}] {};
	\end{tikzpicture}\]

	\begin{enumerate}[label={(\alph*)}]
		\item If $m_{ik}=\infty$, then by transitivity of $F_i$ there exists some permutation $f\in F_i$ mapping $\lambda_i(d')$ to $\lambda_i(d)$. We extend this permutation to an element $g\in\U$ mapping $d'$ to $d$. By the induction hypothesis, all $J$-residues containing a chamber $k$-adjacent to $d'$ are equivalent to $\R$, so it follows from $\U$-invariance that all $J$-residues containing a chamber $k$-adjacent to $g\acts d'=d$ are equivalent to $g\acts\R=\R$ as well.
		\item If $m_{ik}=2$, then all $J$-residues containing a chamber $k$-adjacent to $d$, contain a chamber $k$-adjacent to $d'$ as well. Hence there is nothing to prove.
	\end{enumerate}
	In conclusion, by induction, all $J$-residues adjacent to $\R$ are equivalent to $\R$.

	Repeating the previous arguments (or from another induction on the distance in~$\Gamma$) it follows that all $J$-residues are $\approx$-equivalent, so that $\approx$ is the universal relation. This proves that $\U$ acts primitively on $\Res_J(\Delta)$.
\end{proof}


\nocite{*}
\footnotesize
\bibliographystyle{alpha}
\bibliography{sources}

\end{document}